\documentclass[12pt]{article}

\usepackage[utf8]{inputenc}
\usepackage[T1]{fontenc}
\usepackage{amsthm}
\usepackage{eucal}
\usepackage{dsfont}
\usepackage{mathrsfs}
\usepackage{stmaryrd}
\usepackage{amsmath}
\usepackage{amsfonts}
\usepackage{fancyhdr}
\usepackage{amssymb}
\usepackage{xcolor}
\definecolor{Prune}{RGB}{99,0,60}
\usepackage{mdframed}
\usepackage{multirow} 
\usepackage{multicol} 
\usepackage{scrextend} 
\usepackage{tikz}
\usepackage{graphicx}
\usepackage[absolute]{textpos}
\usepackage{colortbl}
\usepackage{array}
\RequirePackage{geometry}
\usepackage{geometry}
\usepackage[main=english]{babel}
\usepackage[final]{microtype}
\usepackage{lmodern}
\usepackage{textcomp}
\usepackage{pict2e}
\usepackage{enumitem}
\usepackage{hyperref}
\usepackage[nodayofweek]{datetime}
\usepackage{comment}

\hypersetup{            pdfpagemode={UseOutlines},            bookmarksopen,               pdfstartview={FitH},                  colorlinks,                     linkcolor={black},                 citecolor={black},                  urlcolor={black}                               }

\newdateformat{monthyear}{\monthname[\THEMONTH], \THEYEAR}

\newcommand{\N}{\ensuremath{\mathbb{N}}}
\newcommand{\R}{\ensuremath{\mathbb{R}}}

\theoremstyle{plain}
\newtheorem{theorem}{Theorem}[section]
\newtheorem{lemma}[theorem]{Lemma}
\newtheorem{prop}[theorem]{Proposition}

\theoremstyle{definition}
\newtheorem{remark}{Remark}[section]

\newcommand{\mysection}{\setcounter{equation}{0} \section}

\usepackage{authblk}

\def\0{{\mathbf{0}}}

\renewcommand{\P}{\mathbb{P}}

\newcommand{\E}{\mathbb{E}}

\title{\textbf{Poisson process and sharp 
constants in  $L^p $ and Schauder estimates 
for a class of degenerate
Kolmogorov operators
}}

\author[1, 2,*]{\textbf{L. Marino}}
\author[1, 3,$\dagger $]{\textbf{S. Menozzi}}
\author[2,$\ddagger$]{\textbf{E. Priola}}

\affil[1]{\footnotesize{ Laboratoire de Mod\'elisation Math\'ematique d'Evry (LaMME), UMR CNRS 8071,
 Universit\'e Paris-Saclay, Universit\'e d'Evry Val d'Essonne, $23$
Boulevard de France $91037$ Evry, France.}}
\affil[2]{\footnotesize Dipartimento di Matematica, Universit\`a di Pavia, Via Adolfo Ferrata $5$, $27100$ Pavia,
Italy.}
\affil[3]{\footnotesize Laboratory of Stochastic Analysis, Higher School of Economics (HSE), Pokrovsky Boulevard, 11, Moscow, Russian Federation.}

\affil[*]{\footnotesize{lorenzo.marino@univ-evry.fr}}
\affil[$\dagger$]{\footnotesize{stephane.menozzi@univ-evry.fr}}
\affil[$\ddagger$]{\footnotesize{enrico.priola@unipv.it}}

\begin{document}
\maketitle
\begin{abstract}
We  consider a possibly degenerate Kolmogorov-Ornstein-Uhlenbeck operator of the form
   $L={\rm Tr}(BD^2)+\langle Az,D\rangle $, where  $A$, $B $ are $N\times N $ matrices, $z \in \R^N$, $N\ge 1 $, which satisfy the Kalman condition which is equivalent to the  hypoellipticity condition. We prove the following stability result:  the    Schauder and Sobolev estimates associated with the corresponding parabolic Cauchy problem remain valid, with the same constant, for the parabolic Cauchy problem associated with a second order perturbation of $L$, namely for $L+{\rm Tr}(S(t) D^2) $  where  $S(t)$ is a {\it non-negative definite}   $N\times N $ matrix depending continuously on $t \in [0,T]$.
 Our approach relies on the perturbative technique based on the Poisson process introduced in \cite{Krylov:Priola17}.
\end{abstract}


{\small{\textbf{MSC:} Primary: $35$K$10$, $35$K$15$, $60$J$76$.   }}

\section{Introduction}
Let us first consider the following parabolic Cauchy problem:
\begin{equation}
\label{ko_K}
\begin{cases}
   \partial_t u(t,x,y)\,  = \, \Delta_{x}u(t,x,y) + x\cdot \nabla_{y} u (t,x,y)   + f(t,x,y),\\
   u(0, x,y) \, = \, 0,
\end{cases}
\end{equation}
where  $(t,x,y)$ is in $(0, T)\times \R^{2d}$,  \textcolor{black}{for an integer  $d\ge 1$.}
The underlying differential operator
\[
L^{\text{K}}\, = \, \Delta_{x} +x\cdot \nabla_{y} =  \sum_{i=1}^{d}   \partial_{x_i x_i}^2 \, + \, \sum_{i=1}^{d}  \, x_i  \, \partial_{y_i}
\]
is the so-called Kolmogorov operator whose fundamental solution
was derived in the seminal paper \cite{kolm:34}. This particular operator was also mentioned  by H\"ormander as the starting point for his  theory of hypoelliptic operators \cite{Hormander67}.\newline
Let us  write $z = (x,y) \in  \R^{2d}$ and by $\partial_{z_j}$ and $\partial_{z_i z_j}^2$ we denote respectively \textcolor{black}{the} first and \textcolor{black}{the} second partial derivatives with $i,j = 1, \ldots, 2d$.

We are interested in studying the influence of a second order perturbation on equation \eqref{ko_K}.  Precisely, for a time-dependent \textcolor{black}{matrix $\{S(t)\colon t\in [0,T]\}$
  in $\R^{2d}\otimes \R^{2d}$ such that} $t\mapsto S(t) $ is continuous and $S(t) $ is {\it symmetric} and \emph{non-negative definite} for any fixed $t$,  we consider the \textit{perturbed} Cauchy problem:
\begin{equation}
\label{ko_K_Pert}
\begin{cases}
\partial_t u_S(t, z) \, = \, L^{\text{K}}u_S(t,z) + \sum_{i,j=1}^{2d} S_{ij}(t)\,  \partial_{z_i z_j}^2  u_S(t,z)   + f(t,z)
\\
\, \quad \;\;\; \;\;\;\;\; \quad  =: \, L^{\text{K},S}\, u_S(t,z)+f(t,z), \\
u_S(0, z) =0, \;\; z \in \R^{2d}.
\end{cases}
\end{equation}
\textcolor{black}{In particular}, we \textcolor{black}{will} show that Sobolev  (and Schauder) estimates which hold for solutions $u$ \textcolor{black}{of the Cauchy Problem} \eqref{ko_K} are also true, with the same constants,
for solutions $u_S $ to \eqref{ko_K_Pert}. Clearly, \textcolor{black}{the operator} $L^{\text{K}, S}$ \textcolor{black}{can be seen as a} perturbation of $L^{\text{K}}$ involving second order partial derivatives with continuous  time-dependent coefficients.


{For now}, let us explain our \textcolor{black}{main}  results
in a  special form
for equation \eqref{ko_K} in the case of $L^p$-estimates (or Sobolev estimates). \textcolor{black}{For a statement of our results in the whole generality}, we instead refer to Section 2.\newline
For a fixed final time $T>0$ and a source $f$ in $C_0^\infty((0,T)\times \R^{2d})$,
it is known from the work of Bramanti \textit{et al.}  \cite{bram:ceru:manf:96}, \textcolor{black}{Theorem $3.1$} (see also Section \ref{KH_OU} below), that \textcolor{black}{equation} \eqref{ko_K} admits a unique
classical bounded solution $u$ which
satisfies for \textcolor{black}{$p$ in $(1,+\infty)$}  the \textcolor{black}{following} estimates:
\begin{equation}\label{c_K}
 \| \Delta_{x} u \|_{L^p ((0,T) \times \R^{2d})} \le C_p\;  \| f \|_{L^p ((0,T) \times \R^{2d})} = C_p   \| \partial_t u - L^{\text{K}}  u \|_{L^p ((0,T) \times \R^{2d})}.
\end{equation}
 Note that  in this case  $C_p = C_p(d) >0$.
We will actually manage to prove that the unique classical bounded  solution $u_S$ to \eqref{ko_K_Pert} satisfies the estimate
\begin{equation}\label{c_K_PERT}
 \| \Delta_{x} u_S \|_{L^p ((0,T) \times \R^{2d})} \le C_p\;  \| f \|_{L^p ((0,T) \times \R^{2d})} = C_p   \| \partial_t u_S - L^{\text{K},S}  u_S \|_{L^p ((0,T) \times \R^{2d})},
\end{equation}
with the \textbf{same} previous constant $C_p$ as in \eqref{c_K}. This result seems to be new   even in dimension $2d=2$ and even if we only consider
$S(t) = S$, $ t\in [0,T]$,   where $S$  is a $2 \times 2$ {\it symmetric non-negative definite} matrix.

For a uniformly elliptic second order perturbation  $S(t) = S$, $t\in [0,T]$,   where $S$  is \emph{positive definite},  we could also have appealed to  \cite{Bramanti:Cupini:Lanconelli:Priola10}
to derive \textcolor{black}{estimates like in} \eqref{c_K_PERT}. For related estimates in the uniformly elliptic case, see also Section $4$ in Metafune \textit{et al.} \cite{meta:prus:rhan:schn:02}.  However, note that
 from \cite{Bramanti:Cupini:Lanconelli:Priola10}  and \cite{meta:prus:rhan:schn:02}  we could  only deduce that
the constant $C_p$
depends on the ellipticity constant of the perturbation  (this is the first eigenvalue $\lambda_S $ of $S$)
  and   on the maximum eigenvalue of $S$ (on this respect,  see also
    \cite{krylov2003} and  \cite{PriolaINDAM15}).

The remarkable point in \eqref{c_K_PERT} is that the $L^p$-estimates are stable under second order perturbations, which can be  possibly degenerate. Namely,
the fact that $S(t)$ might be degenerate for some $t$ in $(0,T)$, \textcolor{black}{or} even \textcolor{black}{in some} non-empty sub-intervals of $(0,T)$, does not affect the estimates in \eqref{c_K_PERT}.

To prove \eqref{c_K_PERT}, we
combine the results  of \cite{bram:ceru:manf:96}
with a probabilistic \textit{perturbative} approach based on the Poisson process inspired by \cite{Krylov:Priola17}.
There, it was established in particular that the $L^p$-estimates for
 non-degenerate parabolic heat equations with \textcolor{black}{space homogeneous}
 coefficients are
 valid with constants \textcolor{black}{that are} independent of the dimension.


\vskip 1mm

\begin{remark}
Importantly, the approach of \cite{Krylov:Priola17} turns out to be sufficiently robust to handle the estimates in the degenerate directions as well. We recall that the associated maximal $L^p$-regularity
was studied e.g. in \cite{bouc:02}, \cite{Huang:Menozzi:Priola19} or \cite{chen:zhan:19}. Let $p$ in $(1,+\infty)$, there exists $\tilde C_p>0$ such that for $f$ in $C_0^\infty((0,T)\times \R^{2d})$  the  unique classical bounded solution $u$ of \eqref{ko_K} verifies
\begin{equation}\label{LP_KOLMO_DEG}
  \| (\Delta_{y})^{\frac 13}u \|_{L^p ((0,T) \times \R^{2d})} \le \tilde C_p \;  \| f \|_{L^p ((0,T) \times \R^{2d})}=\tilde C_p   \| \partial_t u - L^{\text{K}}  u \|_{L^p ((0,T) \times \R^{2d})},
 \end{equation}
where $(\Delta_{y})^{\frac 13}$ denotes the fractional Laplacian with respect to the degenerate variables $y$ in $\R^{d}$. It turns out that this estimate is also stable for the previously described second order perturbation. Namely, for $u_S $ solving \eqref{ko_K_Pert},
\begin{equation}\label{LP_KOLMO_DEG_PERT}
 \| (\Delta_{y})^{\frac 13}u_S \|_{L^p ((0,T) \times \R^{2d})} \le \tilde C_p \;  \| f \|_{L^p ((0,T) \times \R^{2d})}=\tilde C_p   \| \partial_t u_S - L^{\text{K},S}  u_S \|_{L^p ((0,T) \times \R^{2d})},
 \end{equation}
where again $\tilde C_p $ is the same as in \eqref{LP_KOLMO_DEG}. \qed
\end{remark}


\vskip 1mm
\begin{remark}
The same type of stability results  will also hold for the corresponding  global Schauder estimates, first established in the framework of anisotropic H\"older spaces for the solution of \eqref{ko_K} by Lunardi \cite{Lunardi97} (see also \cite{Marino20}, \cite{Marino21} and the references
 therein).  We refer to estimate \eqref{SCHAU_OU_PERT}.
 \qed
\end{remark}


 We point out that our results in Section 3 could  also {possibly} be  obtained by using the general  theorems  of   Section 4 in \cite{Krylov:Priola17}. This section in \cite{Krylov:Priola17} introduces a more
 general probabilistic approach and  provides unexpected     regularity results. However    checking in our case all the assumptions given in  that  section   is quite involved. On the other hand,      we  provide    self-contained proofs   inspired by  Sections 2 and 3 of \cite{Krylov:Priola17}.

It remains a challenging open problem to have a purely analytic proof of our regularity results.


\paragraph {\bf $L^p$-estimates for degenerate Ornstein-Uhlenbeck operators.}
Let us now describe \textcolor{black}{the more general framework we are going to consider here}.
\textcolor{black}{Let}  $\R^N = \R^{d_0} \times \R^{d_1}$ where $d_0, d_1 $ are two {\it non-negative}   integers such that $d_0 + d_1 =N$  and $d_0 \ge 1$.
let us introduce the non-negative, symmetric matrix \textcolor{black}{$B$ in $\R^N\otimes\R^N$ given by}
$$ B = \begin{pmatrix}
 B_0  & 0\\
 0 & 0
\end{pmatrix} ,$$ where $B_0$ is a symmetric, positive definite matrix in $\R^{d_0}\otimes \R^{d_0}$ such that
 $$ \nu  \sum_{i=1}^{d_0}  {\xi_i^2}
  \le \sum_{i,j =1}^{d_0}
  (B_0)_{ij} \xi_i \xi_j
  \le \frac{1}{\nu} \sum_{i=1}^{d_0}  {\xi_i^2},
 $$
for all $\xi \in \R^{d_0}$, for some $\nu >0$.

We will use, as underlying  \textit{proxy} operators, the family of degenerate Ornstein-Uhlenbeck generators of the form
\begin{equation}
\label{DEF_OU_OP_PROXY}
L^\text{ou} f(z) = 
{\rm Tr}(B D^2 f(z))
+ \langle A z , D f(z)\rangle,\;\;\;  \; z =(x,y) \in \R^{d_0 + d_1}= \R^N,
\end{equation}
for a matrix  $A$ in $\R^N\otimes \R^N$, where $\langle \cdot , \cdot \rangle$
denote\textcolor{black}{s} the usual inner product in $\R^N$.
\textcolor{black}{Moreover}, we assume  the Kalman condition:
\begin{trivlist}
\item[\textbf{[K]}] There exists a  non-negative integer $k$,
 such that
\begin{align} \label{kal}
{\rm Rank} [ B,AB,\ldots,
  A^{k} B
] =N,
 \end{align}
 where $ [B,AB,...,A^{k}B]$  is the $\R^N\otimes \R^{N(k+1)}$  matrix whose
 blocks
 are $B,AB,$ $\ldots , A^kB$.
 From the non-degeneracy of $B_0$, the above condition amounts to say that the vectors
\begin{align} \label{kal_EQUIV}
 \{ e_1 , \ldots, e_{ d_0}, A e_1 ,
 \ldots, Ae_{ d_0}, \ldots, A^k e_1 , \ldots, A^k
 e_{ d_0}\} \;\;\; \mbox{generate} \;\; \R^N,
 \end{align}
 \end{trivlist}
where $\{e_i\}_{i\in \{ 1,\cdots,d_0\} }$ are the first $d_0$ vectors of  the canonical basis for $\R^N$.

 Assumption \textbf{[K]} (which \textcolor{black}{also often} appears  in control theory; see e.g.\ \cite{book:Zabczyk95}) is equivalent to the  H\"ormander condition on the
commutators (\textcolor{black}{c.f.} \cite{Hormander67})  ensuring the hypoellipticity of the operator $\partial_t-L^\text{ou}$. In particular, it implies the existence and \textcolor{black}{the} smoothness of a distributional solution \textcolor{black}{for the following equation:}
\begin{equation}
\label{eq:OU_initial:intro}
\begin{cases}
    \partial_t u(t,z) \, = \,  L^\text{ou} u(t,z) +f(t,z), &\mbox{ on }(0,T)\times \R^N;\\
   u(0,z)\, = \, 0, &\mbox{ on } \R^N,
\end{cases}
\end{equation}
 where $f$ is a function in $C_0^\infty((0,T)\times \R^N)$.

Similarly to  \cite{Krylov:Priola17}, we will prove below the existence and  uniqueness of bounded regular solutions to \eqref{eq:OU_initial:intro}
 assuming that the source $f$ belongs to the space  $B_b\left(0,T;C^\infty_0(\R^N)\right)$, which contains $C_0^\infty((0,T)\times \R^N)$, \textcolor{black}{and that} \textcolor{black}{ can be roughly described as the family of functions which are bounded measurable in time and compactly supported in space \textcolor{black}{uniformly} in time} (see Section \ref{SEC_NOT} for a precise definition). Equation  \eqref{eq:OU_initial:intro}
 will be \textcolor{black}{understood in an} integral form (cf. formula \eqref{int2}).

 By Theorem 3 in \cite{Bramanti:Cupini:Lanconelli:Priola10}  \textcolor{black}{and exploiting some explicit properties of the underlying heat kernel (see Section \ref{KH_OU} below), it can be derived that}
 for \textcolor{black}{any fixed $p$ in $(1,+\infty)$,} there exists  $C_p = C_p(\nu, A, d_0,d_1, T)$  such that
 \begin{equation}\label{de2}
\| D^2_x u \|_{L^p ((0,T) \times \R^N)}\, \le \,  C_p  \| \partial_t u -  L^\text{ou} u    \|_{L^p ((0,T) \times \R^N)}\, = \, C_p  \| f    \|_{L^p ((0,T) \times \R^N)},
\end{equation}
\textcolor{black}{where} for any $z$ in $\R^N$, $t \in [0,T]$, $D_x^2u(t,z)$ \textcolor{black}{stands for the Hessian matrix in $\R^{d_0}\otimes\R^{d_0}$ with respect to the variable $x$}.  We set
$$B_I = \begin{pmatrix}
 I_{d_0,d_0}  & 0_{d_0,d_1}\\
 0_{d_1,d_0} & 0_{d_1,d_1}
\end{pmatrix}$$
  and
  \textcolor{black}{note, in particular,} that \eqref{de2} can be rewritten
in the following, equivalent way:
\begin{equation}\label{de3}
\begin{split}
\| B_I D^2 u \, B_I \|_{L^p ((0,T) \times \R^N)} \, =
\, \|  D^2_x u \,  \|_{L^p ((0,T) \times \R^N)} \\
\le \,  C_p \| \partial_tu -   L^\text{ou} u    \|_{L^p ((0,T) \times \R^N)}
=\, C_p  \| f    \|_{L^p ((0,T) \times \R^N)},
\end{split}
\end{equation}
\textcolor{black}{where $D^2 u = D^2_z u$ represents instead the full Hessian matrix in $\R^N\otimes\R^N$ with respect to $z$}.

\textcolor{black}{Fixed} a continuous mapping $t\mapsto S(t)$ \textcolor{black}{such that} $S(t)$ is a {\it symmetric} and \emph{non-negative definite} \textcolor{black}{matrix in $\R^N\otimes\R^N$},  $t \in [0,T],$ we consider \textcolor{black}{again}
the following perturbation of $L^\text{ou}:$
\begin{equation} \label{ciao}
\begin{split}
L_t^{\text{ou},S} f(z) \, &:=\,  
{\rm Tr}(B D^2 f(z)) +
 {\rm Tr}(S(t) D^2 f(z))
+ \langle A z , D f(z)\rangle\\
&= \, L^\text{ou} f(z)+
{\rm Tr}(S(t) D^2 f(z)),
\end{split}
\end{equation}
where $z=(x,y)$ is in $\R^{d_0+d_1}=\R^N$.
For the solution $u_S$  of the related Cauchy problem
\begin{equation}
\label{eq:OU_PERT0}
\begin{cases}
    \partial_t u_S(t,z) \, = \,  L_t^{\text{ou},S} u_S(t,z) +f(t,z), &\mbox{ on }(0,T)\times \R^N;\\
   u_S(0,z)\, = \, 0, &\mbox{ on } \R^N,
\end{cases}
\end{equation}
we will prove the following main theorem:

\begin{theorem} \label{d44}  Let us consider \eqref{eq:OU_PERT0} with $f \in B_b\left(0,T;C^\infty_0(\R^N)\right) $. \textcolor{black}{Then,} there exists a unique solution $u_S$ \textcolor{black}{of Cauchy Problem \eqref{eq:OU_PERT0}}
 which verifies,
with the same constant $C_p$ as in \eqref{de3},
\begin{gather}\label{de4}
\|  D^2_x u_S \, \|_{L^p ((0,T) \times \R^N)}
=\| B_I D^2 u_S \, {B_I} \|_{L^p ((0,T) \times \R^N)}
\\ \nonumber \le C_p  \;  \| \partial_t u_S -   {L}_t^{{\rm{ou}},S} u_S  
\|_{L^p ((0,T) \times \R^N)}=C_p  \;  \| f
\|_{L^p ((0,T) \times \R^N)}.
\end{gather}
\end{theorem}
We point out  that for time-homogeneous non-negative definite matrices $S$,
 the  corresponding  elliptic $L^p$-estimates  as in formula (5) of \cite{Bramanti:Cupini:Lanconelli:Priola10}   (replacing $\mathcal A$ in \cite{Bramanti:Cupini:Lanconelli:Priola10} with  $ L^{\text{ou},S}: =$ $
{\rm Tr}(B D^2 \cdot ) +
 {\rm Tr}(S D^2 \cdot  )
+ \langle A z , D \cdot \rangle
)$ {\sl with a constant independent of $S$,}  could also be derived from \eqref{de4}  using an argument  given in \cite{Bramanti:Cupini:Lanconelli:Priola10}.

 For more information on the OU operator $L^{ou}$ we also  refer to  the recent work by Fornaro \textit{et al.} \cite{forn:meta:pall:schn:21} about  full description of the spectrum of degenerate OU operators in $L^p$-spaces.

{Independently from the constant preservation, we also emphasize that the $L^p$-estimates in \eqref{de4} for the perturbed operator seem, to the best of our knowledge, to be new and have some interest by their own.}

    {Let us eventually mention that  our stability results could turn out to be useful to investigate the well-posedness of some related stochastic differential equations through the corresponding martingale problem}.

      We could  actually derive more general estimates, \textcolor{black}{possibly} involving the degenerate directions as well, \textcolor{black}{dependingly on} the structure of $A$. \textcolor{black}{Some results in that direction are gathered in Section \ref{ESTENSIONI}}. Anyhow, to illustrate our approach we now briefly present the various steps to derive  \eqref{de4}.


\subsection{Strategy of the proof for estimate \eqref{de4}.}  \label{SEZ_STRAT}
\textcolor{black}{Fixed a classical bounded solution $u$ to Cauchy Problem \eqref{eq:OU_initial:intro}}, let us introduce $v(t,z) := u(t, e^{-tA} z)$.  This \textcolor{black}{well-known}  transformation (cf. \cite{daprato-Lunardi}) precisely allows to get rid of the drift term in the PDE satisfied by  $v$.
Indeed, we have \textcolor{black}{that}  $u(t,z) = v(t, e^{tA} z)$
and since $ u$ solves \eqref{eq:OU_initial:intro}, it holds for any $(t,z)$ in $(0,T)\times \R^N$, that:
 \begin{equation}
 \begin{split}
 f {(t,z)}\, &= \,  \partial_t u(t,z) {\color{black} - L^{\text{ou}}}  u(t,z)\\
 &= \, v_t(t,e^{tA} z) + \langle D v
 (t, e^{tA} z) , A e^{tA} z\rangle  {\color{black} -} {\text Tr} \big( e^{tA} B
 e^{tA^*} D^2  v(t, e^{tA} z) \big) \\
& \qquad - \langle D v
 (t, e^{tA} z) , A e^{tA} z\rangle\\
&= \, v_t(t,e^{tA} z)    {\color{black} -} {\text Tr} \big( e^{tA} B
 e^{tA^*} D^2 v(t, e^{tA} z) \big).
 \end{split}
 \label{COMP_0H_VAR}
\end{equation}
\textcolor{black}{Denoting $\tilde{f}(t,z):= f(t,e^{-tA}z)$,}
 it now follows that $v$ satisfies the PDE:
 \begin{equation} \label{ma}
 \begin{cases}
 \partial_tv(t, z)   \, = \,  {\text Tr} \big( e^{tA} B
 e^{tA^*} D^2 v(t, z) \big) + \tilde f(t,z) &\mbox{ on }(0,T)\times \R^N;\\
 v(0,z) =0 &\mbox{ on } \R^N.
 \end{cases}
\end{equation}
In terms of the function $v$, the known  estimates in \eqref{de3} rewrites as:
\begin{equation} \label{1}
 \| {B_I} e^{ t A^*} D^2 v (t, e^{tA} \cdot ) \, e^{tA} {B_I} \|_{L^p((0,T)\times \R^N)} \le C_p  \|  \tilde f(t, e^{tA}  \cdot )  \|_{L^p((0,T)\times \R^N)},
\end{equation}
where we used the notation  $\| {B_I} e^{- t A^*} D^2 v (t, e^{tA} \cdot ) \, e^{tA} {B_I} \|_{L^p((0,T)\times \R^N)} $  to stress the dependence on $t$
 instead of the more precise  formulation
$$
{ \| {B_I} e^{ \cdot \,  A^*} D^2 v (\cdot , e^{ \cdot \, A} \cdot ) \, e^{ \cdot \, A} {B_I} \|_{L^p((0,T)\times \R^N)}}.
$$
By changing variable in the integrals,
 control \eqref{1}  is equivalent to
\begin{equation}\label{2}
\| {B_I} e^{tA^*} D^2 v (t,  \cdot ) \, e^{tA} {B_I} \|_{L^p((0,T)\times \R^N,{m})} \le C_p  \|  \tilde f  \|_{L^p((0,T)\times \R^N, {m})},
\end{equation}
{where $L^p((0,T)\times \R^N,m)$ denotes the $L^p$-norms w.r.t. the measure $m(dt,dx)= {\rm det}(e^{-At}) dt dx $}.

Considering now the following more general Cauchy problem on
 $[0,T] \times \R^N$
\begin {equation} \label{d2}
\begin{cases}
 \partial_t w(t, z)    =   {\text Tr} \big( e^{tA} B
 e^{tA^*} D^2 w(t, z) \big) + {\text Tr} \big( e^{tA} S(t)
 e^{tA^*} D^2 w(t, z) \big) +  \tilde f(t,z);\\
 w(0,z)=0,
 \end{cases}
\end{equation}
we  can  establish \textcolor{black}{the well-posedness of} the Cauchy problem \eqref{d2}, \textcolor{black}{exploting}, for instance, probabilistic arguments, using  the underlying Gaussian process.

Now  the crucial  step consists in adapting  some arguments from \cite{Krylov:Priola17} based on the use of the Poisson process  to derive that the same
 $L^p$-estimates in \eqref{2} still hold for $w$, independently \textcolor{black}{from} the non-negative definite, symmetric matrices  $S(t)$. Precisely,
\begin{equation} \label{11}
\| {B_I} e^{tA^*} D^2 w (t, \cdot ) \, e^{tA} {B_I} \|_{L^p((0,T)\times \R^N, {m})} \le C_p \|  \tilde f(t,  \cdot )  \|_{L^p((0,T)\times \R^N, {m})},
\end{equation}
with the \emph{same} constant $C_p$ \textcolor{black}{appearing in \eqref{2}}.

The last step then consists in  coming back  to the Ornstein-Uhlenbeck operators \textcolor{black}{framework}.  Namely, \textcolor{black}{we introduce} $\tilde u(t,z) := w(t, e^{tA} z)$ \textcolor{black}{which} solves, by definition, the following equation:
 \begin{equation*}
\begin{cases}
 \partial_t \tilde u(t,z) = L_t^{\text{ou}, S} \tilde u(t,z)  + f  (t,z),\ (t,z)\in (0,T)\times \R^N,\\
 \tilde u(0,z)=0,\ z\in \R^N.
\end{cases}
\end{equation*}
  Thus  $\tilde u = u_S$.
 {Noticing} that $D^2 w(t, \cdot) = D^2 [\tilde u(t, e^{-tA} \, \cdot )] $ $= e^{-tA^*}D^2 \tilde u(t, e^{-tA} \cdot )e^{-tA} $  we thus get from \eqref{11} that the following estimates hold:
\begin{equation}\label{de311}
\| {B_I} D^2 \tilde u \, {B_I} \|_{L^p ((0,T) \times \R^N)} \, \le \, 
C_p  \|  f \|_{L^p ((0,T) \times \R^N)} .
\end{equation}
Through the previous steps we have then constructed a solution $\tilde u $ \textcolor{black}{of Cauchy Problem} \eqref{eq:OU_PERT0} which indeed satisfies \textcolor{black}{the estimates in} \eqref{de4} with the same $C_p$, associated with the \textit{unperturbed} or \textit{proxy} operator.
 The maximum principle will eventually provide uniqueness for the solution $\tilde u $.

  \vskip 1mm
  \begin{remark} i) We point out that we could also consider more  general time-dependent Ornstein-Uhlenbeck operators like:
  \[   M  \, = \, {\rm Tr}(B(t) D^2 \cdot)
+ \langle A z , D \cdot \rangle.\]
Arguing as before starting from $L^p$-estimates (or Schauder estimates) for $M$ we can derive the same $L^p$-estimates (or Schauder estimates) for a perturbation of $M$ like \eqref{ciao}.

ii) We could  extend the $L^p$-estimates (or the Schauder estimates) related to  $L^{{\rm ou}}$ to more general
 operators like
$$
L_t^{\text{ou},S} f(z) +  \langle b(t), D f(z) \rangle
$$
where $b: \R_+ \to \R^N$ is continuous.  We can   even add to $L_t$  a  possibly degenerate   non-local perturbation  (cf. Section 7 of \cite{Krylov:Priola17}). The $L^p$-estimates (or Schauder estimates) are  still   preserved with the same constant.
For the sake of simplicity in the sequel we will only consider $b(t)=0$ and we will  not deal with  non-local perturbations of $L_t^{\text{ou},S}.$
  \end{remark}

\paragraph{Organization of the paper.}
The article is organized as follows. At the end of the current section, we first fix  some useful notations.
In Section \ref{SEC_DRIFT_LESS} we will then focus on driftless second order Cauchy problems associated with a non-negative definite, possibly degenerate, diffusion matrix. We will also consider its   relation to the Ornstein-Uhlenbeck dynamics.  We will establish through the probabilistic perturbation approach of \cite{Krylov:Priola17} that if some $L^p$-estimate holds  for a particular diffusion matrix so does it, with the same associated constant  as explained before, for a non-negative perturbation of the diffusion matrix (see Section 3).
Finally, by the arguments of Section 1.1 we will obtain \eqref{de311}.  Stability results in anisotropic
Sobolev space and Schauder estimates are given in Section 4.



\subsection{Definition  of solution and useful  notations}\label{SEC_NOT}


%

Let us consider the following Cauchy problem:
\begin{equation}
\label{eq:Cauchy_problem_gen0}
\begin{cases}
\partial_t v(t,z)   =  \text{tr}\left(Q(t)D^2 v(t,z)\right) +
\langle  b(t,z), Dv(t,z) \rangle
+f(t,z), \mbox{ on }(0,T)\times\R^N;\\
v(0,z)  =  0, \mbox{ on } \R^N;
\end{cases}
\end{equation}
where $Q\colon [0,T]\to \R^N\otimes \R^N$ is a {\it   continuous  symmetric non-negative definite matrix} and $b: [0,T] \times \R^N \to \R^N$ is a {\it continuous function} such that  $|b(t,z)| \le K_T (1+ |z|)$, $(t,z) \in [0,T] \times \R^N$, for some constant $K_T>0$.

The function $f$ belongs to  $B_b\left(0,T;C^\infty_0(\R^N)\right)$, the space of all Borel bounded functions $\phi \colon [0,T]\times\R^N \to \R$ such
that $\phi(t,\cdot)$ is smooth and compactly supported for any $t$ in $[0,T]$; for any $n$ in $\N$ the $C^n(\R^N)$-norms of $\phi(t,\cdot)$ are bounded in time and the supports of the functions $\phi(t,\cdot)$ are contained in the same ball. Moreover, we require that, for any $z \in \R^N$, the mapping:
\begin{equation}\label{dd}
t \mapsto \phi(t,z)
\end{equation}
is a \emph{piece-wise continuous} function on $[0,T]$, i.e.\ it is continuous except \textcolor{black}{for} a finite number of points.

\vskip 1mm

\begin{remark}
 Note that to perform the technique used in \cite{Krylov:Priola17} and based on the Poisson process we
need to consider equations like \eqref{eq:Cauchy_problem_gen0} with a source $f$ which is possibly discontinuous in time (cf. the proof in Section 2 of \cite{Krylov:Priola17} and Section \ref{SEC_PERTURBATIVO} below).
\end{remark}

We interpret \textcolor{black}{Cauchy Problem \eqref{eq:Cauchy_problem_gen0}} in  an {\it integral} form:
\begin{equation}\label{int2}
v(t,z)=\int_0^t \Big(f(s,z)+{\rm Tr}(Q(s) D^2 v(s,z) ) +  \langle  b(s,z), Dv(s,z) \rangle\Big) ds.
\end{equation}
\textcolor{black}{In particular,} we say that a continuous and bounded function $v : [0, T] \times \R^N \to \R$ is a solution to equation \eqref{eq:Cauchy_problem_gen0} if $v(t, \cdot)$ belongs to  $C^{2}(\R^N)$, for any $t \in [0,T]$, and \eqref{int2} holds as well,  for any $(t,z) \in [0, T] \times \R^N$.

We finally note that, for any $z \in \R^N$, the function $t \mapsto v(t,z)$ is a $C^1$-piece-wise   function on $[0,T]$.


 By Theorem 4.1 in \cite{KP10} we  deduce in a quite standard way   that if a solution $v$ exists then it is unique and the following maximum principle holds:
\begin{equation}\label{max}
 \sup_{(t,z)\in [0,T] \times \R^N}|v(t,z)|\le T\,  \sup_{(t,z)\in [0,T]\times \R^N}|f(t,z)|.
\end{equation}
 About the proof of \eqref{max} we only make some remarks.
 By considering $v$ and $-v$ we see that it is enough to prove
 that $v(t,z) \le T \| f\|_{\infty}$, for all $(t,z)\in [0,T] \times \R^N$. Moreover, setting $\tilde v = $ $
  v - t \| f\|_{\infty}$, we note that  $\tilde v$ verifies  \eqref{int2} with $f$ replaced by $f - \| f\|_{\infty} \le 0$.
   Finally, by considering the equation verified by  $e^{- t} \tilde v$, we can apply Theorem 4.1 in \cite{KP10}
   to obtain the result.

\mysection{Estimates for driftess second order operators and related perturbation}
\label{SEC_DRIFT_LESS}
Throughout this section, we consider the following Cauchy problem:
\begin{equation}
\label{eq:Cauchy_problem_gen}
\begin{cases}
\partial_t v(t,z) \, = \, \text{Tr}\left(Q(t)D^2 v(t,z)\right) +f(t,z) &\mbox{ on }(0,T)\times\R^N;\\
v(0,z) \, = \, 0 &\mbox{ on } \R^N,
\end{cases}
\end{equation}
 which can be seen as  a special case of \eqref{eq:Cauchy_problem_gen0} when $b=0$. Moreover, we assume that $Q$ is not identically zero.

\subsection {Well-posedness }
\begin{prop}[Well-posedness in integral form for the driftless Cauchy problem]\label{PROP_WP_DRIFTLESS}
\textcolor{black}{Let $f$ be in $B_b\left(0,T;C^\infty_0(\R^N)\right)$. Then,} there exists a unique
 solution $v$ to Cauchy problem \eqref{eq:Cauchy_problem_gen} \textcolor{black}{in an integral sense}, i.e.,  it solves for  $(t,z)\in [0,T]\times \R^N$:
\begin{equation}
\label{INTEGRAL}
v(t,z)=\int_0^t \Big(f(s,z)+{\rm Tr}(Q(s) D^2 v(s,z)\Big)ds.
\end{equation}
We will denote in short $v=PDE(Q,f)$.

\end{prop}
\begin{proof}
By the maximum principle  (cf. equation \eqref{max}) uniqueness  holds for Cauchy Problem \eqref{eq:Cauchy_problem_gen}. \textcolor{black}{We can then focus on proving the existence of a solution.}
\textcolor{black}{Let us introduce now}  \[v(t,z) \, := \, \int_0^t \mathbb E[f(s,z+I_{s,t})] \, ds\]
with the following notation: $I_{s,u}:=\sqrt 2\int_s^u Q(r)^{1/2}dW_r $, where $W$ is an $N$-dimensional Brownian motion on some probability space $(\Omega,\mathcal F,(\mathcal F_t)_{t\ge 0},\mathbb P) $ and $Q(r)^{1/2}$ stands for a square root of $Q(r)$, i.e. $Q(r)=Q(r)^{1/2}(Q(r)^{1/2})^* $.


\textcolor{black}{Applying the  It\^o formula in space to $f(s, z+I_{s,u} )_{u\in [s,t]}$}, we get that
$$
\E f(s, z+  I_{s,t} ) = f(s,z)+ \mathbb E \Big[\int_s^{t} {\rm{Tr}}(Q(u) D^2 f(s,z+I_{s,u})){}du \Big].
$$
Hence,
$$
v(t,z)=\int_{0}^t \Big(f(s,z)+\mathbb E[\int_s^{t} {\rm{Tr}}(Q(u) D^2 f(s,z+I_{s,u})){}du]\Big)ds,$$
from which it readily follows that
\[
\begin{split}
\partial_t v(t,z) \, &= \, f(t,z)+\int_0^t  \mathbb E[{\rm{Tr}}(Q(t) D^2f(s,z+I_{s,t}))]ds\\
&= \, f(t,z)+{\rm{Tr}}\Big(Q(t) D^2\int_0^t  \mathbb E[ f(s,z+I_{s,t})]ds\Big)\\
&= \, f(t,z)+{\rm{Tr}}\Big(Q(t) D^2v(t,z)\Big).
\end{split}
\]
\textcolor{black}{for almost every $t\in [0,T]$ and any $z \in\R^N $}.
\end{proof}

\subsection{Relation to the Ornstein-Uhlenbeck dynamics}\label{REL_OU}
If now in particular, $Q(t)$ has the particular form   $Q(t)= e^{tA} B
 e^{tA^*}$ (cf.\ equation \eqref{ma}), we introduce
 $$
 u(t,z):=v(t,e^{tA}z),
 $$
  where $v$ is the solution to \eqref{INTEGRAL}   (see Proposition \ref{PROP_WP_DRIFTLESS}).
 Since we can differentiate with respect to $t$ the function $u(\cdot,z)$ for a.e. $t \in [0,T]$, we can perform
 computations similar to \eqref{COMP_0H_VAR} and get  that $ u (t,z)$  solves in integral form:
\begin{equation}
\label{eq:Cauchy_pro}
\begin{cases}
\partial_t u(t,z) \, = \,  L^\text{ou} \, u(t,z) +\bar f(t,z), &\mbox{ on }(0,T)\times\R^N;\\
u(0,z) \, = \, 0, &\mbox{ on } \R^N;
\end{cases}
\end{equation}
with $L^\text{ou}$ as in  \eqref{DEF_OU_OP_PROXY},
$\bar f(t,z)=f(t,e^{tA}z) $.  Precisely, for all $(t,z)\in [0,T]\times\R^N $,
\begin{equation}
\label{INTEGRAL_OU}
u(t,z)=\int_0^t \Big(\bar f(s,z)+ L^\text{ou} v(s,z)\Big) ds.
\end{equation}
We have  that $u$ is a solution to \eqref{eq:Cauchy_pro}.

 Let us also point out that the well-posedness of \eqref{eq:Cauchy_pro} could also have been obtained directly from Gaussian type calculations,
similar to those in the proof of Proposition \ref{PROP_WP_DRIFTLESS},
 introducing $u^{{\rm ou}}(t,z):=\int_0^t \mathbb E[\bar f(s,e^{(t-s)A}z+I_{s,t}^{{\rm ou}})] ds$ where $I_{s,u}^{{\rm ou}}:=\sqrt 2\int_s^u \textcolor{black}{e^{(u-v)A}} B dW_v  $.

\subsection{About the $L^p$-estimate \eqref{de2} for the OU operator}
\label{KH_OU}
{The aim of this section is to} fully justify the estimates in \eqref{de2}.  This is a consequence of the previous {probabilistic} representation and of Theorem 3 in \cite{Bramanti:Cupini:Lanconelli:Priola10}. For $u$ solving \eqref{eq:OU_initial:intro} it holds
that for all $(t,z)\in [0,T]\times \R^N $,
\begin{equation}\label{PROB_REP_OU}
u(t,z)=\int_0^t \mathbb E[ f(s,e^{A(t-s)}z+I_{s,t}^{{\rm ou}})] ds=\int_{0}^t\int_{\R^N } f(s,z') p^{{\rm ou}}(t-s,z,z') dz' ds,
\end{equation}
where for $v>0$, $p^{{\rm ou}}(v,z,\cdot) $ stands for the density at time $v$ of the stochastic process
\[
X^{{\rm ou}}_u \,:= \,e^{Au}z+\sqrt 2\int_0^u \textcolor{black}{e^{A(u-w)}} B dW_w\, = \, z+\int_0^u A X^{{\rm ou}}_w dw+ \sqrt{2}BW_u,\, u\ge 0.
\]
We recall from \cite{Lanconelli:Polidoro94} that assumption \textbf{[K]} is equivalent to the fact that there exists $k\in \N $ and positive integers $(\mathfrak d_i)_{i\in\{1,\cdots, k\}}$ s.t. $\sum_{i=1}^k \mathfrak d_i=d_1 $ and for all $i\in \{1,\cdots,k\} $, setting $\mathfrak d_0=d_0 $ and $\sum_{m=0}^{-1} =0$, the matrixes
$${\mathscr A}^i\, := \, (A_{j,\ell})_{(j,\ell)\in \{\sum_{m=0}^{i-1}\mathfrak d_m+1,\cdots,\sum_{m=0}^{i}\mathfrak d_m\}\times \{\sum_{m=1}^{i-1}\mathfrak d_m +1,\cdots,\sum_{m=1}^{i}\mathfrak d_m\}},$$
have rank $\mathfrak d_i $.  {The matrix $A$ writes:
\begin{align} \label{sotto}
A \, = \,
    \begin{pmatrix}
        \ast   & \ast  & \dots  & \dots  & \ast   \\
         {\mathscr A}^1  & \ast  & \ddots & \ddots  & \vdots   \\
        0_{\mathfrak d_2,d_0}      & {\mathscr A}^2   & \ast  & \ddots & \vdots \\
        \vdots &\ddots & \ddots& \ddots & \ast \\
        0_{\mathfrak d_k,d_0}      & \dots & 0_{\mathfrak d_k,\mathfrak d_{k-1}}     & {\mathscr A}^{k}    & \ast
    \end{pmatrix}.
\end{align}
}
Following the proof of Lemma 5.5 in \cite{dela:meno:10}, where the case $d_0=d,\ \mathfrak d_i=d, k=n-1$ is addressed, it can be derived that there exists $C\ge 1$ s.t. for all $(v,z,z')\in (0,T]\times (\R^N)^2 $, 
\begin{equation}
\label{EST_DENS}
|D_{x}^2p^{{\rm ou}}(v,z,z')|\, \le \, \frac{C}{v^{\sum_{i=0}^{k-1}\mathfrak d_i (i+\frac 12)+1}}\exp\left(-C^{-1}v|\mathbb T_v^{-1}(e^{Av}z-z')|^2 \right),
\end{equation}
where $\mathfrak d_0 =d_0$ and
$$\mathbb T_v \, := \, \text{diag}(vI_{d_0\times d_0},v^2 I_{d_1\times d_1},\dots,v^{ {k+1}}I_{d_{k}\times d_{k}}), \quad v\ge0,$$
reflects the various scales of the system. For a given function $f\in B_b\left(0,T;C^\infty_0(\R^N)\right)$, it is then clear from \eqref{PROB_REP_OU} and \eqref{EST_DENS} that for all $(t,z)\in (0,T]\times \R^N $:
\begin{equation}\label{REP_SING}
D_{x}^2 u(t,z)\, = \, {\rm p.v.}\int_{0}^t\int_{\R^N } f(s,z')D_x^2 p^{{\rm ou}}(t-s,z,z') dz' ds.
\end{equation}
It indeed suffices to observe that:
\begin{gather*}
|{\rm p.v.}\int_{0}^t\int_{\R^N } f(s,z')D_x^2 p^{{\rm ou}}(t-s,z,z') dz' ds|\\
=|{\rm p.v.}\int_{0}^t\int_{\R^N } [f(s,z')-f(s,e^{A(t-s)}z)]D_x^2 p^{{\rm ou}}(t-s,z,z') dz' ds|  \\
\underset{\eqref{EST_DENS}}{\le} \sup_{s\in [0,T]}\|Df(s,\cdot)\|_\infty    \\ \times \int_{0}^t\int_{\R^N }\frac{C}{(t-s)^{\sum_{i=0}^{k-1}\mathfrak d_i (i+\frac 12)+\frac 12}}\exp\left(-C^{-1}(t-s)|\mathbb T_{t-s}^{-1}(e^{A(t-s)}z-z')|^2 \right)dz'ds\\
\le C \sup_{s\in [0,T]}\|Df(s,\cdot)\|_\infty T^{\frac 12}.
\end{gather*}
The estimates in \eqref{de2}
 now follows from the proof of Theorem 3 in \cite{Bramanti:Cupini:Lanconelli:Priola10}, starting from \eqref{REP_SING} instead of (16) therein. The strategy is clear. {It is necessary to}  introduce a cut-off function which separates the points $(s,z') $ which do not induce any singularity in \eqref{REP_SING} for the derivatives of the density, namely such that $t-s\ge c_0 $ or $|e^{A(t-s)}z-z'|\ge c_0 $, for some fixed constant $c_0>0$, from those who are close to the singularity.  For the non-singular part of the integral the expected $L^p$-control readily follows from \eqref{EST_DENS} and the Young inequality (see also Proposition 5 in \cite{Bramanti:Cupini:Lanconelli:Priola10}), whereas the derivation of the bound for the singular part requires some involved harmonic analysis, see Section 4  {on} the same reference.
  We can also refer to Theorem 11 and its proof in \cite{prio:15} for similar issues {linked with} the corresponding $L^p$-estimates for degenerate Ornstein-Uhlenbeck {operators} in an elliptic setting.

\subsection{The main result  for equation \eqref{eq:Cauchy_problem_gen}: perturbation of second order driftless PDE }

{Let us fix $p$ in $(1,+\infty)$ and  assume that} there exists $R(t) \in \R^N \otimes \R^N$ depending continuously on $t \ge 0$ and a constant $C_p>0$, such that for any  $f$ in $B_b\left(0,T;C^\infty_0(\R^N)\right)$,
the unique solution $v = PDE(Q,f)$ to {equation} \eqref{eq:Cauchy_problem_gen} satisfies
\begin{equation}\label{s223}
 \| R(t)^* D^2 v  \, R(t) \,  \|_{L^p((0,T)\times \R^N, {\mathfrak m})} \le C_p  \|   f  \|_{L^p((0,T)\times \R^N,{\mathfrak m})},
\end{equation}
{for some absolutely continuous measure $\mathfrak m $ w.r.t. the Lebesgue measure on $ [0,T]\times \R^N$  such that  $\mathfrak m(dt,dx)=g(t)dtdx $ for some Borel bounded function $g$
(note that in \eqref{2} we have $R(t) =e^{tA} {B_I} $, $\mathfrak m(dt,dx)=g(t)dtdx={\rm det}(e^{-At}) dtdx $)}.

We would like to {exhibit that a control} like \eqref{s223} {also} holds for the solution $w$ to {the following Cauchy Problem:}
\begin{equation}\label{w}
\begin{cases}
\partial_t w(t,z)  =  \text{tr}\left(Q(t)D^2 w(t,z)\right)+\text{tr}\left(Q'(t) D^2 w(t,z)\right)+f(t,z), \mbox{ on }(0,T)\times \R^N;\\
w(0,z)  = 0, \mbox{ on }\R^N,
\end{cases}
\end{equation}
Namely we have to prove \textcolor{black}{the following result}.

\begin{theorem}\label{uno} Let us consider equations  \eqref{eq:Cauchy_problem_gen} and   \eqref{w}  where   $Q(t)$, $Q'(t)$ are two  continuous in time, non-negative definite matrices in $\R^N\otimes \R^N$ and   $f \in B_b\left(0,T;C^\infty_0(\R^N)\right)$. Assume that   estimate \eqref{s223} holds as explained above.
 Then the solution $w$ to \eqref{w} verifies
\begin{equation}\label{s224}
 \| R(t)^* D^2 w  \, R(t) \,  \|_{L^p((0,T)\times \R^N, {\mathfrak m})} \le C_p  \|   f  \|_{L^p((0,T)\times \R^N, {\mathfrak m})},
\end{equation}
$p \in (1, \infty)$ with the same constant $C_p$ as in \eqref{s223}.
\end{theorem}

From Theorem \ref{uno} using the argument of Section 1.1 we can easily derive Theorem \ref{d44}.

\mysection{A perturbation  argument for  proving Theorem \ref{uno}}
We aim here at applying the probabilistic perturbative approach considered in \cite{Krylov:Priola17}. The key  idea in that work was, for a well-posed PDE which enjoys some quantitative given estimates, to introduce a \textit{small} random perturbation in the source $f$  through a suitable  Poisson type process  and to investigate the properties of the associated PDE involving an unknown function $v$.
 After considering  a  {small} random perturbation of $v$,
 we arrive at the useful integral formula \eqref{eq:1_OU}.  Taking the expectation
 the contributions  associated with the jumps yield, for an appropriate intensity of the underlying Poisson process, a finite difference operator.
 For the PDE satisfied by the expectation,  involving the finite difference operator, {\it the  initial estimates are preserved}.
 Repeating the previous argument we can obtain a PDE  involving the composition of two  finite difference operators.

 Compactness arguments  then allow to derive that, the initial estimates still hold at the limit with the  composition of two finite difference operators  replaced by the corresponding differential \textcolor{black}{operator} of order two. Iterating  this procedure we can obtain the result.

Below, we start recalling  basic properties of  Poisson type processes and   corresponding stochastic integrals, which  are needed for our approach.

\subsection{Poisson stochastic integrals}

We {briefly} recall here the very definition of {the stochastic}
 integral     {driven by a Poisson process}.
We start reminding the construction of {such} processes.
\newline
Given a probability space $(\Omega,\mathcal{F},\mathbb{P})$ to be fixed from this point further,  we start considering a sequence of independent real-valued random variables $\{\tau_n\}_{n \in \N}$ on $\Omega$ whose distribution is exponential of parameter $\lambda>0$:
\[\mathbb{P}(\tau_n>r) \,= \, e^{-r\lambda},  \quad r\ge 0.\]
We can then define the partial sums sequence $\{\sigma_n\}_{n \in \N}$  as follows:
\begin{align*}
    \sigma_0\, = \, 0; \;\;\;
    \sigma_n  =   \sum_{i=1}^n \tau_i, \quad n=1,2, \dots
\end{align*}
\textcolor{black}{For any fixed} $t\ge 0$, $\pi_t$ now \textcolor{black}{denotes} the number of consecutive sums of $\tau_i$ which lie on $[0,t]$, i.e.
\begin{equation}
\label{poisson}
  \pi_t=\sum_{n=0}^{\infty} \mathds{1}_{\sigma_n \le t},
\end{equation}
where $\mathds{1}_{\sigma_n \le t}$ represents the indicator function of the event $\{\sigma_n\le t\}$. The process $\{\pi_t\}_{t\ge 0}$ we have just constructed is usually known in the literature as a \emph{Poisson process} with intensity $\lambda$ (see, for instance, \cite{book:Protter04}).

Now, let $c : [0,T] \to \R^N$ be a continuous function.  We can define the  Poisson stochastic integral as
\begin{gather} \label{s25}
b_t \, :=  \,\int_0^t c (s) d\pi_s \,= \, \sum_{\sigma_k \le t,\;\; k \ge 1}  {c(\sigma_k)}= \sum_{0 < s \le t}  {c(s)} (\pi_s - \pi_{s-}),\;\;\; t \in [0,T],
\end{gather}
 $b_0=0$ (as usual $\pi_{s-}(\omega)$ denotes the left limit at $s$, for any $\omega$, $\P$-a.s.).  We now recall the following formula for the expectation of the stochastic integral:
\begin{equation}\label{see}
\E [\int_0^t c (s) d\pi_s]  = \lambda \int_0^t c (s) ds \in \R^N.
\end{equation}
 (cf.  Lemma 2.1  in \cite{Krylov:Priola17} for a direct proof; see also   Theorem 16 in \cite{book:Protter04} and Theorem 5.3 in \cite{Krylov:Priola17} for a more general formula involving stochastic integrals of  predictable processes against the Poisson process).
We also recall  the following more general  result.

\vskip 1mm

\begin{lemma} \label{LEMME_POISSON}
Let $\{\pi_t\}_{t\ge0}$ be a Poisson Process of intensity $\lambda$ on a probability space $(\Omega,\mathcal{F},\mathbb{P}) $.
Let us consider a stochastic process $ (\xi_t )_{t \in [0,T]} $ with values in $\R$ which has c\`adl\`ag  paths ($\P$-a.s.) and is ${\cal F}_t$-adapted where
 ${\cal F}_t$ is the  completed $\sigma$-algebra generated by the random variables $\pi_s$, $0 \le s \le t$.  Suppose that $\sup_{\omega \in \Omega, \; s \in [0,T]} |\xi_ s(\omega)| < \infty $. Then
 \begin{equation}\label{rt}
\mathbb E\int_0^t \xi_{s-} \, d\pi_s=\lambda \int_0^t \mathbb E\xi_s ds.
\end{equation}
\end{lemma}

%
%

\subsection{Proof of Theorem \ref{uno}}
 \label{SEC_PERTURBATIVO}

 According to the notations in  Proposition \ref{PROP_WP_DRIFTLESS}, let  $v=PDE(Q ,f)$ and  $w=PDE(Q+Q' ,f) $ be  the unique solutions
 of \textcolor{black}{equations} \eqref{eq:Cauchy_problem_gen} and \eqref{w}, respectively.\newline
 \textcolor{black}{The proof of Theorem \ref{uno} will be obtained adapting} the method developed in \cite{Krylov:Priola17} (see in particular Section $3$, therein).
Let $e_1$ be the first unit vector in $\R^N$. We define
$$
 X_t = \int_{0}^t \sqrt{Q'(r)} \, e_1 d \pi_r
$$
where $\sqrt{Q'(t)}$ is the unique $N \times N$ symmetric non-negative definite square root of $Q'(t)$ and  $\{\pi_t\}_{t\ge0}$ is a Poisson Process of intensity $\lambda$ (cf.
 \eqref{s25}). The parameter $\lambda$ will be chosen {appropriately} later on.

Recall that the solution  $v$ to \eqref{eq:Cauchy_problem_gen} is given by
\begin{equation}\label{s55}
v(t,z) = \int_0^t ds\int_{\R^N} [f(s,z+ z') \mu_{\textcolor{black}{s,t}} (dz')
\end{equation}
where $\mu_{\textcolor{black}{s,t}}$ is the Gaussian law of the stochastic integral $I_{s,t}:=\sqrt 2\int_s^t \sqrt{Q(r)} \, dW_r$ (see the proof of Proposition \ref{PROP_WP_DRIFTLESS}).

Let us fix  $\epsilon>0$.  We notice   that the shifted source $f_\epsilon(t,z):=f(t,z-\epsilon X_t)$ (which also depends on $\omega$; we have omitted to write  such dependence on $\omega$) is again in $B_b\left(0,T;C^\infty_0(\R^N)\right)$.  \textcolor{black}{This is} the reason why we have considered such a function space for the source. \textcolor{black}{It precisely allows to take into account the time discontinuities coming from the jumps of the Poisson process}.

For any fixed $\omega$ in $\Omega$, Proposition \ref{PROP_WP_DRIFTLESS} readily gives that there exists a unique  solution $v_\epsilon=$  PDE$(Q,f(t,z-\epsilon X_t))$, depending  on $\epsilon$ and  $\omega$ as parameters, such that
\begin{align}
\label{proof:estimate1}
\sup_{(t,z)\in [0,T]\times \R^N}| v_\epsilon(t,z)| \, &\le \, T \sup_{(t,z)\in [0,T]\times \R^N}| f(t,z)|.
\end{align}
Moreover,  thanks to the invariance for translations of the $L^p$-norms, \textcolor{black}{it follows from \eqref{s223}} that
\begin{equation}\label{w11}
\| R(t)^* D^2 v_{\epsilon}  \, R(t) \,  \|_{L^p((0,T)\times \R^N,\textcolor{black}{\mathfrak m})} \le C_p  \|   f_{\epsilon}  \|_{L^p((0,T)\times \R^N,\textcolor{black}{\mathfrak m})} = C_p \|   f  \|_{L^p((0,T)\times \R^N,\textcolor{black}{\mathfrak m})}.
\end{equation}
 By \textcolor{black}{Equation} \eqref{s55}, we know that $v_{\epsilon} $ is given by
\begin{gather*}
v_{\epsilon}(t,z) = \int_0^t ds\int_{\R^N} [f(s,z - \epsilon X_s + z') \mu_{\textcolor{black}{s,t}} (dz').
\end{gather*}
For each $z \in \R^N$, the stochastic process  $ (\textcolor{black}{v_{\epsilon}}(t,z))_{t \in [0,T]} $ has continuous paths ($\P$-a.s.) and it is ${\cal F}_t$-adapted where
 ${\cal F}_t$ is the completed  $\sigma$-algebra generated by the random variables $\pi_s$, $0 \le s \le t$.

For  fixed $z \in \R^N$, $\epsilon >0,$ let us introduce the process $(\textcolor{black}{v_\epsilon}(t,z+\epsilon X_t))_{t \in [0,T]}$ which is given by
$$
 v_{\epsilon}(t,z+\epsilon X_t) = \int_0^t ds\int_{\R^N} [f(s,z +\epsilon X_t - \epsilon X_s + z') \mu_{\textcolor{black}{s,t}} (dz').
$$
\textcolor{black}{It is not difficut to check that it is}   ${\cal F}_t$-adapted and  \textcolor{black}{it has  c\`adl\`ag paths}.

Applying \eqref{INTEGRAL} on each interval $[\sigma_n,\sigma_{n+1}\wedge t ), n\in \{0,\cdots, \pi_t] $ on which $X_s$ is constant,  one then derives that:
\begin{equation}
\label{eq:1_OU}
v_\epsilon(t,z+\epsilon X_t) \, = \, \int_0^t\left(\text{tr}(Q(s)D^2 v_\epsilon(s,z+\epsilon X_s))+ f(s,z)\right)\, ds +\int_0^t g_\epsilon(s,z) \,d\pi_s,
\end{equation}
where $g_\epsilon(s,z)=v_\epsilon(s,z+  \epsilon \sqrt{Q'(s)} \, e_1 + \epsilon X_{s-})-v_\epsilon(s,z+\epsilon X_{s-})$ is precisely the contribution associated with the jump  times. It is clear that $g_\epsilon(s,z)\neq 0$ if and only if $\pi_s$ has a jump at time $s$.
We then have by Lemma \ref{LEMME_POISSON}:
\begin{equation}
\label{proof:eq1}
    \mathbb{E}\int_0^tg_\epsilon(s,z)\, d\pi_s \, = \, \lambda \int_0^t\left(\bar v_\epsilon(s,z+\epsilon \sqrt{Q'(s)} \, e_1)- \bar v_\epsilon(s,z)\right)\, ds,
\end{equation}
where $\bar v_\epsilon(s,z) = \mathbb{E}[v_\epsilon(s,z+\epsilon X_s)]$. Let \textcolor{black}{us denote}
$$
l(t) := \sqrt{Q'(t)} \, e_1.
$$
Taking the expectation on both sides of equation \eqref{eq:1_OU}, we find out that $\bar v_\epsilon$ is an integral solution of the following PDE:
\begin{equation}
\label{eq:2_OU}
    \partial_t \bar v_\epsilon(t,z) \, = \, \text{tr}(Q(t)D^2 \bar v_\epsilon(t,z))+\lambda\left(\bar v_\epsilon(t,z+\epsilon l(t))-\bar v_\epsilon(t,z)\right)+f(t,z),
\end{equation}
with zero initial condition.  Remark that uniqueness of bounded continuous solutions to \eqref{eq:2_OU} follows by the maximum principle,
arguing
as in the proof of Lemma 2.2 in \cite{Krylov:Priola17} (first one considers the case  $ \lambda T \le 1/4$ and then one iterates the procedure by steps of size $1/(4 \lambda)$).

Moreover by \eqref{w11} we obtain (using also the Jensen inequality and  the Fubini theorem)
\begin{gather*}
\| R(t)^* D^2 \bar v_{\epsilon}  \, R(t) \,  \|_{L^p((0,T)\times \R^N,\textcolor{black}{\mathfrak m})}^p
= \int_{(0,T) \times \R^N} |R(t)^* D^2 \bar v_{\epsilon} (t,z) \, R(t)|^p \textcolor{black}{\mathfrak m(dt,dz)}
\\
= \int_{(0,T) \times \R^N} | \E [R(t)^* D^2 v_{\epsilon} (t,z + \epsilon X_t) R(t)] \,|^p dz g(t)dt
\\
\le   \int_{[0,T] \times \R^N} \E [ | R(t)^* D^2 v_{\epsilon} (t,z + \epsilon X_t) R(t) \,|^p] dz g(t)dt
\\
= \E \int_{[0,T] \times \R^N}   | R(t)^* D^2 v_{\epsilon} (t,z + \epsilon X_t) R(t) \,|^p dz g(t)dt
\\
 = \textcolor{black}{\mathbb E}\int_{[0,T] \times \R^N}  | R(t)^* D^2 v_{\epsilon} (t,\bar z ) R(t) \,|^p d\bar z g(t)dt
  \\
\le C_p^p  \|   f  \|_{L^p((0,T)\times \R^N,\textcolor{black}{\mathfrak m})}^p,
\end{gather*}
using \eqref{w11} for the last inequality \textcolor{black}{($L^p$-estimate for the PDE with random source)}.
Choosing $\lambda = \epsilon^{-2}$ we have \textcolor{black}{
from \eqref{eq:2_OU}
}
 \begin{equation}
\label{eq:2_OUe}
    \partial_t \bar v_\epsilon(t,z) \, = \, \text{tr}(Q(t)D^2\bar v_\epsilon(t,z))+\ \epsilon^{-2}\left(\bar v(t,z+\epsilon l(t))-\bar v_\epsilon(t,z)\right)+f(t,z),
\end{equation}
with zero initial condition and moreover
\begin{equation}\label{d3}
\| R(t)^* D^2 \bar v_{\epsilon}  \, R(t) \,  \|_{L^p((0,T)\times \R^N, {\mathfrak m} )}^p  \le C_p^p  \|   f  \|_{L^p((0,T)\times \R^N, {\mathfrak m} )}^p.
\end{equation}
{Now the idea is to apply again the same reasoning above to the equation \eqref{eq:2_OUe} with respect to $\bar v_\epsilon$, using   $f(t,z+\epsilon X_t)$}  again with  $\lambda=\epsilon^{-2}$. We obtain first a solution $p_{\epsilon}$ to \eqref{eq:2_OUe} corresponding to $f(t,z+\epsilon X_t)$ and then \textcolor{black}{derive that}
$$
w_{\epsilon}(t,z) = \E [p_{\epsilon }(t , z - \epsilon X_t)]
$$
is the  unique bounded continuous (integral) solution $w_\epsilon$ of the following problem:
\begin{equation}
    \label{eq:3_OU}
\begin{cases}
\partial_t w_\epsilon(t,z) \,  = \, \text{tr}(Q(t)D^2 w_\epsilon(t,z)) \\ +\epsilon^{-2}\left(w_\epsilon(t,z+\epsilon l(t))-2w_\epsilon(t,z)+w_\epsilon(t,z-\epsilon l(t))\right) +     f(t,z),\\
w_\epsilon(0,z) \, = \, 0.
\end{cases}
\end{equation}
  The previous estimates  still hold  with $w_\epsilon$ instead of $v_\epsilon$, i.e.,
\begin{gather}
\label{proof:estimate3_BIS}
\sup_{(t,z)\in[0,T]\times \R^N} |w_\epsilon(t,z)| \, \le \, T\sup_{(t,z)\in[0,T]\times \R^N} |f(t,z)|;
\\
\label{proof:estimate4_BIS}
    \| R(t)^* D^2 w_{\epsilon}  \, R(t) \,  \|_{L^p((0,T)\times \R^N,\textcolor{black}{\mathfrak m})}  \le C_p  \|   f  \|_{L^p((0,T)\times \R^N,\textcolor{black}{\mathfrak m})}.
\end{gather}
We would like now to let $\epsilon$ goes to zero,
possibly passing to a subsequence $\epsilon_n \to 0$,
and prove that the associated limit $w$ solves
\begin{equation}
    \label{eq:3_OUf}
\begin{cases}
\partial_t w(t,z) \, = \, \text{tr}(Q(t)D^2 w(t,z))+ \langle D^2 w (t,z) \sqrt{Q'(t)}e_1,  \sqrt{Q'(t)}e_1 \rangle +f(t,z),\\
w(0,z) \, = \, 0
\end{cases}
\end{equation}
and  estimates \eqref{proof:estimate3_BIS} and \eqref{proof:estimate4_BIS} \textcolor{black}{hold} with $w_{\epsilon}$ replaced by $w$.

    To do so we will proceed by compactness. Namely, we are going to prove that the family of solutions $w_\epsilon$ solving \eqref{eq:3_OU}, indexed by the parameter $\epsilon$, is equi-Lipschitz on any compact subset of $[0,T]\times \R^N$ and the same holds for
  any derivative in space of $w_\epsilon$. Indeed, one can apply the finite difference operators with respect to $z$ at any order in \eqref{eq:3_OU}.
We recall that  for a \textit{smooth} function $\phi\colon \R^N\to \R$, the first finite difference $\delta_{h,i}\phi$, $i\in \{1,\cdots,N\} $ of step $h >0$ in the direction $e_i$ ($i^{\rm th}$ basis vector) is
given by
\[\delta_{h,i}\phi(z) \, = \, \frac{\phi(z+he_i)-\phi(z)}{h}, \quad z \in \R^N.\]For a given multi-index $\gamma \in \N^N$,  the $\gamma$-th order finite difference operator $\delta_{h,\gamma}$, is then defined, for any $h>0$, through composition. Namely,
\[\delta_{h,\gamma}\phi(z) \, = \, \delta^{\gamma_1}_{h,1}\delta^{\gamma_2}_{h,2}\dots \delta^{\gamma_N}_{h,N}\phi(z),\]
where $\delta^{\gamma_i}_{h,i}$ {denotes} the $\gamma_i$-th times composition of $\delta_{h,i}$ with itself.
\newline
 Since  any spatial derivative of $f$ belongs to $B_b\left(0,T;C^\infty_0(\R^N)\right) $,
 using  \eqref{proof:estimate3_BIS}   we deduce first that any finite difference of any order of $ w_\epsilon$ is bounded. Consequently, $ w_\epsilon$ is infinitely differentiable in space with bounded derivatives on $[0,T]\times \R^N $. Equation \eqref{eq:3_OU}, to be understood in its integral form similarly to \eqref{INTEGRAL}, then gives that those derivatives are themselves Lipschitz continuous in time (uniformly in the space variable).  This precisely gives the equi-Lipschitz on any compact subset of $[0,T]\times \R^N$ of the family $w_\epsilon$ and of any spatial derivative of $w_{\epsilon}$.
\newline
We can now apply the Arzel\`a-Ascoli theorem to $w_\epsilon$ showing the existence of a sub-sequence $\{w_{\epsilon_n}\}_{n \in \N}$ which converges uniformly on any compact set to a function $w\colon [0,T]\times \R^N\to \R$. Similarly,
 any  derivative in space of $w_{\epsilon_n}$ tends to the respective derivatives of $w$, uniformly on the compact sets.\newline
Passing to the limit as $n\to \infty$ along the sequence $(\epsilon_n)_n $  in equation \eqref{eq:3_OU}  (written in the  integral form), we can then conclude that $w$ solves \eqref{eq:3_OUf}.

Moreover,  estimates \eqref{proof:estimate3_BIS} and \eqref{proof:estimate4_BIS} holds with $w_{\epsilon}$ replaced by $w$.
 Iterating the previous argument in $N$  steps we finally prove that the unique solution $w$ to
 \begin{equation}
    \label{eq:3_OUff}
\begin{cases}
\partial_t w(t,z)  =  \text{tr}(Q(t)D^2 w(t,z))+  \sum_{k=1}^N\langle D^2 w (t,z) \sqrt{Q'(t)}e_k,  \sqrt{Q'(t)}e_k \rangle  +f(t,z),\\
w(0,z) \, = \, 0
\end{cases}
\end{equation}
verifies   estimates \eqref{proof:estimate3_BIS} and \eqref{proof:estimate4_BIS}  with $w_{\epsilon}$ replaced by $w$.
 The proof is complete.  \qed



\def\ciao{

%

\begin{prop}
\label{prop:link_ou_diffusive}
Let $f$ be in $B_b\left(0,T;C^\infty_0(\R^N)\right)$. Then, a function $u$ is the unique bounded, continuous solution of \textcolor{black}{the} Cauchy Problem \eqref{eq:def_ou_operator} if and only if the function $\tilde{u}$ given by $\tilde{u}(t,x):=u(t,e^{tA}x)$ is the unique bounded, continuous solution of the following:
\begin{equation}
\label{eq:PDE}
    \begin{cases}
    \partial(t) \tilde{u}(t,x) \, = \, \text{tr}\left(e^{-tA}Be^{-tA^\ast}D^2_x\tilde{u}(t,x)\right) +\tilde{f}(t,x), &(t,x)\in (0,T]\times \R^N;\\
    \tilde{u}(0,x) \, = \, 0, &x\in \R^N;
    \end{cases}
\end{equation}
where the function $\tilde{f}$ is given by $\tilde{f}(t,x)=f(t,e^{tA}x)$.
\end{prop}

The above result links the Cauchy Problem \eqref{eq:def_ou_operator}
involving Ornstein-Uhlenbeck operators with the diffusion one \eqref{eq:Cauchy_problem_gen}. In particular, it allows us to exploit Theorem \ref{thm:general} for the  Ornstein-Uhlenbeck operators.

From the assumed estimate on $u$ in \eqref{A_PRIO_EST}, we derive that for $\tilde u $ as in Proposition \ref{prop:link_ou_diffusive}:
\begin{align}\label{SHIFTED_A_PRIORI}
[u]_1:=[u(t,x)]_1=[\tilde u(t,e^{-t A}x)]_1=:\leftB \tilde u \rightB_1\le C[f]_2=[\tilde f(t,e^{-t A}x)]_2=:\leftB \tilde f\rightB_2,
\end{align}
with $\tilde{f}(t,x)=f(t,e^{tA}x) $ again as in Proposition \ref{prop:link_ou_diffusive}. The idea is now to apply Theorem \ref{thm:general}. To this end we first need to prove that $\leftB\cdot\rightB_1, \leftB\cdot\rightB_2 $ defined in \eqref{SHIFTED_A_PRIORI} are actually suitable semi-norms in the sense of Definition \ref{DEF_SUITABLE_NORM}. This property readily follows from the previous definition and the fact that $[\cdot]_1, [\cdot]_2 $ themselves are suitable seminorms.

It is also clear that, for $S$ as in the theorem, there exists a unique solution integral $u_S$ to \eqref{eq:OU_PERT} which is continuous and twice continuously differentiable in space. Similarly, we can associate to $ u_S$ the mapping $\tilde u_S(t,x)=u_S(t,e^{-At}x) $ which solves:
\begin{equation}
\label{eq:PDE_PERT}
    \begin{cases}
    \partial_t \tilde{u}_S(t,x) \, = \, \text{tr}\left(e^{-tA}(B+S(t))e^{-tA^\ast}D^2_x\tilde{u}_S(t,x)\right) +\tilde{f}(t,x), &\mbox{ on }[0,T]\times \R^N;\\
    \tilde{u}_S(0,x) \, = \, 0, &\mbox{ on }\R^N;
    \end{cases}.
\end{equation}
In short, with the notations of Section \ref{SEC_DRIFT_LESS}, we can write  $\tilde u=PDE(Q(t),\tilde f) $ with $Q(t)=e^{-tA}Be^{-tA^\ast} $ and $\tilde u_S=PDE(Q(t)+Q(t)',\tilde f) $ with $Q(t)'=e^{-tA}(B)e^{-tA^\ast} $. It is clear that $Q(t),Q(t)' $ are non-negative. From \eqref{SHIFTED_A_PRIORI}
and Theorem \ref{thm:general} we thus obtain:
\begin{align}\label{SHIFTED_A_PRIORI_FINAL}
\leftB \tilde u_S \rightB_1\le C\leftB \tilde f\rightB_2,
\end{align}
with the same previous constant $C$. From the definition of $(\leftB\cdot\rightB_i)_{i\in \{1,2\}} $ in \eqref{SHIFTED_A_PRIORI} we eventually derive:
$$[u_S]_1\le C[f]_1, $$
which concludes the proof.

\subsection{Concrete seminorms considered}\label{CONCRETE_SN}
In the final section of our article, we briefly present two natural examples of \emph{suitable} seminorms for which we can apply Theorem \ref{thm:OU_general} when the matrices $A,B$ satisfy the Kalman condition \textbf{[K]}.\newline
In order to introduce them, we firstly need to talk about the anisotropic nature of degenerate Ornstein-Uhlenbeck operators satisfying \textbf{[K]}. Intuitively, the appearance of this kind of phenomena is due essentially to the particular structure of the matrix $A$
that allows the smoothing effect of the second order operator $\text{tr}(BD^2_x)$, acting only on the first "component" given by $B_0$, to propagate into the system with lower and lower intensity. We recall indeed from \cite{Lanconelli:Polidoro94} that for degenerate Ornstein-Uhlenbeck operators, the Kalman condition is equivalent to the fact that $A,B$ must have the form:
\begin{equation}\label{eq:Lancon_Pol}
B \, = \,
    \begin{pmatrix}
        B_0 & 0_{d_1,N-d_1}   \\
        0 _{N-d_1,d_1}&   0_{N-d_1,N-d_1}  \\
    \end{pmatrix}
\,\, \text{ and } \,\, A \, = \,
    \begin{pmatrix}
        \ast   & \ast  & \dots  & \dots  & \ast   \\
         A_2   & \ast  & \ddots & \ddots  & \vdots   \\
        0      & A_3   & \ast  & \ddots & \vdots \\
        \vdots &\ddots & \ddots& \ddots & \ast \\
        0      & \dots & 0     & A_n    & \ast
    \end{pmatrix}
\end{equation}
where $B_0$ is a non-degenerate matrix in $\R^{d_1}\otimes \R^{d_1}$ and $A_h$ are matrices in $\R^{d_h}\otimes \R^{d_{h-1}}$ with
$\text{rank}(A_h)=d_h$ for any $h$ in $\llbracket 2,n\rrbracket$. Moreover, $d_1\ge d_2\ge \dots\ge d_n\ge 1$ and $d_1+\dots+d_n=N$.

\paragraph{Anisotropic Geometry of Dynamics.} \textcolor{black}{Da S. a E.: forse da accorciare o sopprimere, la parte cinetica intendo.}
In order to illustrate how degenerate Ornstein-Uhlenbeck operators usually behave, we focus for the moment on the prototypical example we have already considered in the introduction, the Kolmogorov operator ${\Delta}_{\text{K}}$.\newline
Fixed $N=2d$ and $n=2$ we recall that $\Delta_{\text{K}}$ can be represented for any sufficiently regular function $\phi\colon \R^N\to \R$ as
\[\Delta_{\text{K}}\phi(x) \, = \, \Delta_{x_1}\phi(x)+ x_1\cdot D_{x_2}\phi(x) \quad \text{ on }\R^{2d},\]
where $(x_1,x_2)\in\R^{2d}$ and $\Delta_{x_1}$ is the standard Laplacian with respect to $x_1$. In our framework, it is associated with the matrices
\[A_{\text{K}} \, := \,
    \begin{pmatrix}
               0 & 0 \\
               I_{d\times d} & 0
    \end{pmatrix} \,\, \text{ and } \,\, B \, := \,
    \begin{pmatrix}
                I_{d\times d} & 0 \\
                 0 & 0
    \end{pmatrix}.
             \]
We then search for a dilation
\[\delta_\lambda\colon [0,\infty)\times \R^{2d} \to [0,\infty)\times \R^{2d}\]
which is invariant for the considered dynamics, i.e.\ a dilation that transforms solutions of the equation
\[\partial_tu(t,x)-\Delta_{\text{K}} u(t,x) \, = \, 0 \quad \text{ on }(0,\infty)\times\R^{2d}\]
into other solutions of the same equation. The idea of a dilation operator $\delta_\lambda$ that summarizes the multi-scaled behaviour of the dynamics was firstly introduced by Lanconelli and Polidoro in \cite{Lanconelli:Polidoro94}.  Since then, it has
become a "standard" tool in the analysis of degenerate equations (see e.g. \cite{Lunardi97}, \cite{Huang:Menozzi:Priola19} or  \cite{Marino20}, \cite{Marino21} in the fractional setting). \newline
Due to the particular sub-diagonal structure of $A_{\text{K}}$ and the natural scaling of $\Delta_{x_1}$ of order $2$, it makes sense to consider for any fixed $\lambda>0$, the
following
\[ \delta_\lambda(t,x_1,x_2)\, :=\,  (\lambda^2 t,\lambda x_1,\lambda^3 x_2).\]
It then holds that
\[\bigl(\partial_t -\Delta_{\text{K}}\bigr) u = 0 \, \Longrightarrow \bigl(\partial_t -\Delta_{\text{K}} \bigr)(u \circ
\delta_\lambda) = 0.\]
The previous reasoning suggests to introduce a parabolic distance $\mathbf{d}_P$ that is homogenous with respect to the dilation $\delta_\lambda$,
so that:
\[\mathbf{d}_P\bigl(\delta_\lambda(t,x);\delta_\lambda(s,x')\bigr) = \lambda \mathbf{d}_P\bigl((t,x);(s,x')\bigr).\]
Following the notations in \cite{Huang:Menozzi:Priola19}, we then introduce the distance $\mathbf{d}_P$ on $[0,\infty)\times \R^{2d}$  given by
\begin{equation}\label{Definition_distance_d_P}
\mathbf{d}_P\bigl((t,x),(s,x')\bigr)  \, := \, \vert s-t\vert^\frac{1}{2}+\mathbf{d}(x,x') \, := \,
\vert s-t\vert^\frac{1}{2} +|x_1-x_1'|+|x_2-x_2'|^{\frac{1}{3}}.
\end{equation}
As it can be seen, $\mathbf{d}_P$ is an extension of the standard parabolic distance adapted to respect the multi-scale nature
of the underlying operator. Indeed, the exponents appearing in \eqref{Definition_distance_d_P} are those which make each space component homogeneous to
the characteristic time scale $t^{1/2}$.\newline
From a more
probabilistic point of view, the exponents in Equation \eqref{Definition_distance_d_P} can also be related to the characteristic
time scales of the iterated integrals of a Brownian Motion. It can be easily seen from the example, noticing that the operator $\Delta_{\text{K}}$ corresponds to the generator of the Brownian Motion and its time integral.

Going back to the general framework with $A,B$ as in \eqref{eq:Lancon_Pol}, we can now construct the suitable distance $\mathbf{d}$ associated with the Cauchy Problem \eqref{eq:OU_initial} on $\R^N$.\newline
Recalling that $n$ denotes the smallest integer such that the Kalman rank condition [\textbf{K}] holds, we follow
\cite{Lunardi97} decomposing the space $\R^N$ with respect to the family of linear operators $B, AB,\dots,
A^{n-1}B$. \newline
We start defining the family $\{V_h\colon h\in \llbracket 1,n \rrbracket\}$ of subspaces of $\R^N$ through
\[
V_h \,  := \, \begin{cases}
            \text{Im} (B), & \mbox{if } h=1, \\
            \bigoplus_{k=1}^{h}\text{Im}(A^{k-1}B), & \mbox{otherwise}.
        \end{cases}\]
It is easy to notice that $V_h\neq V_k$ if $k\neq h$ and $V_1\subset V_2\subset\dots V_n=\R^N$. We can then construct iteratively the family
$\{E_h \colon h\in \llbracket 1,n \rrbracket\}$ of orthogonal projections from $\R^N$ as
\[E_h \,  := \,
        \begin{cases}
            \text{projection on } V_1, & \mbox{if } h=1; \\
            \text{projection on }(V_{h-1})^\perp \cap V_h, & \mbox{otherwise}.
        \end{cases}
\]
With a small abuse of notation, we will identify the
projection operators $E_h$ with the corresponding matrices
in $\R^N\otimes \R^N$. Actually, we have that $\dim E_h(\R^N)=d_h$ for any $h\ge 1$, where the integer $d_h$ already appeared in equation \eqref{eq:Lancon_Pol}.\newline
We can then define the distance $\mathbf{d}$ through the decomposition $\R^N=\bigoplus_{h=1}^nE_h(\R^N)$ as
\[\mathbf{d}(x,x') \,:= \, \sum_{h=1}^{n}\vert E_h(x-x')\vert^{\frac{1}{1+2(h-1)}}.\]
We highlight however that, technically speaking, $\mathbf{d}$ does not induce a norm on $[0,T]\times \R^N$ in the usual sense since it lacks of
linear homogeneity. \newline
The anisotropic distance $\mathbf{d}$ can be understood direction-wise: we firstly fix a "direction" $h$ in $\llbracket
1,n\rrbracket$ and then calculate the standard Euclidean distance on the associated subspace $E_h(\R^N)$, but scaled according to the
dilation of the system in that direction. We conclude summing the contributions associated with each component.\newline
Highlighted the particular anisotropic nature of the operator through the distance $\bm{d}$, we can now extend two standard functional spaces, the H\"older and the Sobolev ones, in order to reflect the intrinsic multi-scale behavior given by the dilation operator $\delta_\lambda$.
\textcolor{black}{Da S. a E.: veramente forse dovremmo parlare della parte infatti omogenea che d\`a gli esponenti.}

\paragraph{Anisotropic H\"older spaces.}
Fixed some $k$ in $\N_0:=\N\cup\{0\}$ and $\beta$ in $(0,1]$, we follow Krylov \cite{book:Krylov96_Holder} denoting the Zygmund-H\"older semi-norm for a function $\phi\colon \R^N\to \R$ as
\[[\phi]_{C^{k+\beta}} \, := \,
\begin{cases}
    \sup_{\vert \vartheta \vert= k}\sup_{x\neq y}\frac{\vert D^\vartheta\phi(x)-D^\vartheta\phi(y)\vert}{\vert x-y\vert^\beta} , & \mbox{if
    }\beta \neq 1; \\
    \sup_{\vert \vartheta \vert= k}\sup_{x\neq y}\frac{\bigl{\vert}D^\vartheta\phi(x)+D^\vartheta\phi(y)-2D^\vartheta\phi(\frac{x+y}{2})
    \bigr{\vert}}{\vert x-y \vert}, & \mbox{if } \beta =1.
                             \end{cases}\]
Consequently, The Zygmund-H\"older space $C^{k+\beta}_b(\R^N)$ is the family of functions $\phi\colon \R^N
\to\R$ such that $\phi$ and its derivatives up to order $k$ are continuous and the norm
\[\Vert \phi \Vert_{C^{k+\beta}_b} \,:=\, \sum_{i=1}^{k}\sup_{\vert\vartheta\vert = i}\Vert D^\vartheta\phi
\Vert_{\infty}+[\phi]_{C^{k+\beta}_b} \,\text{ is finite.}\]

We can define now the anisotropic Zygmund-H\"older spaces associated with the distance $\mathbf{d}$. Fixed $\gamma>0$, the space $C^{\gamma}_{b,d}(\R^N)$ is
the family of functions $\phi\colon \R^N\to \R$ such that for any $h$ in $\llbracket 1,n\rrbracket$ and any $x_0$ in $\R^N$, the function
\[z\in  E_h(\R^N)\,  \to \, \phi(x_0+z) \in \R \,\text{ belongs to }C^{\gamma/(1+2(h-1))}_b\left(E_h(\R^N)\right),\]
with a norm bounded by a constant independent from $x_0$. It is endowed with the norm
\begin{equation}\label{eq:def_anistotropic_norm}
\Vert\phi\Vert_{C^{\gamma}_{b,d}} \,:=\,\sum_{h=1}^{n}\sup_{x_0\in \R^N}\Vert\phi(x_0+\cdot)\Vert_{C^{\gamma/(1+2(h-1))}_b}.
\end{equation}
We highlight that it is possible to recover the expected joint regularity for the partial derivatives, when
they exist, as in the standard H\"older spaces. In such a case, they actually turn out to be H\"older continuous with respect to the
distance $\mathbf{d}$ with order one less than the function (See Lemma $2.1$ in \cite{Lunardi97} for more details).
\newline
Exploiting that for the "standard" Zygmund-H\"older norm $\Vert\cdot\Vert_{L^\infty((0,T),C^{\gamma}_{b})} $ is clearly suitable (in the sense of Definition \ref{DEF_SUITABLE_NORM}) and reasoning components by components, it is not difficult to conclude that also the anisotropic norm $\Vert\cdot\Vert_{L^\infty((0,T),C^{\gamma}_{b,d})}$ is suitable as well.\newline
Moreover, Lunardi in \cite{Lunardi97} showed that the following anisotropic Schauder estimates holds for the solution $u$ of Cauchy Problem \eqref{eq:OU_initial}:
\begin{equation}\label{SCHAU_OU}
\Vert u \Vert_{L^\infty((0,T),C^{2+\beta}_{b,d})} \, \le \, C \Vert f \Vert_{L^\infty((0,T),C^{\beta}_{b,d})},
\end{equation}
for some constant $C$ independent from $f$.
Thus, all the assumptions of Theorem \ref{thm:OU_general} are satisfied with $[\cdot]_1 
=\Vert\cdot\Vert_{L^\infty( (0,T),C^{2+\beta}_{b,d})}$ and $[ \cdot ]_2=\Vert\cdot\Vert_{L^\infty((0,T),C^{\beta}_{b,d})}$.

\paragraph{Anisotropic Sobolev spaces.}
Given $\alpha$ in $(0,1)$ and $h$ in $\llbracket 1, n \rrbracket$, we want to introduce the $\alpha$-fractional Laplacian $\Delta^\alpha_{x_h}$ along the "component" $x_h$.
To do it, we follow \cite{Huang:Menozzi:Priola19} by considering the orthogonal projection $p_h\colon \R^N\to \R^{d_h}$ such that $p_h(x)=x_h$ and denoting its adjoint by $B_h=p_h\colon \R^{d_h}\to \R^N$. Notice that $B=B_1B^\ast_1$ in a matricial form.\newline
We can now define the $\alpha$-fractional Laplacian $\Delta^\alpha_{x_h}$ as:
\[\Delta^{\alpha}_{x_h}\phi(x) \, := \, \text{p.v.}\int_{\R^{d_i}}\left[\phi(x+B_hz)-\phi(x)\right] \frac{dz}{|z|^{d_h+\alpha}}, \quad x \in \R^N,\]
for any sufficiently regular function $\phi \colon \R^N\to \R$.\newline
Fixed $p$ in $(1,+\infty)$, we recall that we have denoted by $L^p([0,T]\times \R^N)$ the standard $L^p$-space with respect to the Lebesgue measure.\newline
We can now define the anisotropic Sobolev space associated with the distance $\bm{d}$. For notational simplicity, let us denote \[\alpha_h \,:=\, \frac{1}{1+\alpha(h-1)}.\]
The space $\dot{W}^{2,p}_d([0,T]\times \R^N)$ is composed by all the functions $\varphi\colon [0,T]\times\R^N\to \R$ in $L^p([0,T]\times \R^N)$ such that for any $h$ in $\llbracket 1, n \rrbracket$,
\[\Delta^{\alpha_h}_{x_h}\varphi(t,x) \, := \,  \Delta^{\alpha_h}_{x_h}\varphi(t,\cdot)(x) \text{ belongs to } L^p([0,T]\times \R^N).\]
It is endowed with the natural norm
\[\Vert \varphi \Vert_{\dot{W}^{2,p}} \, = \, \sum_{h=1}^n\Vert \Delta^{\alpha_h}_{x_h}\varphi \Vert_{L^p}.\]
Similarly to the first example, we can reason componentwise in order to show that the anisotropic Sobolev norm $\Vert \cdot \Vert_{\dot{W}^{2,p}_d}$ is indeed a suitable seminorm in our definition.\newline
In \cite{Huang:Menozzi:Priola19}, see also \cite{chen:zhan:19}  and \cite{menozziSPA} where time inhomogeneous coefficients are considered as well, it has been proven that when the strictly upper diagonal elements of $A$ in \eqref{eq:Lancon_Pol} are equal to zero then the following Sobolev estimates hold:
\begin{equation}\label{SOB_EST_GLOB}
\Vert u \Vert_{\dot{W}^{2,p}_d} \, \le \, C_p\Vert f \Vert_{L^p},
\end{equation}
where again $u$ is the unique bounded solution of Cauchy problem \eqref{eq:OU_initial}. In particular we get also the maximal smoothing effects w.r.t. the degenerate directions.
The specific structure assumed on $A$ is actually due
to the fact that for such matrices there is an underlying homogeneous space structure which makes easier to establish maximal regularity estimates (see e.g. \cite{coif:weis:71} in this general setting). Anyhow, in this case the assumptions of Theorem \ref{thm:OU_general} are satisfied with $[\cdot]_1
=\Vert\cdot\Vert_{\dot{W}^{2,p}_d}$ and $
[\cdot]_2=\Vert\cdot\Vert_{L^p}$.

For a general $A$ with the form in \eqref{eq:Lancon_Pol} such that $A$, $B$ satisfy \textbf{[K]}, we can refer to \cite{Bramanti:Cupini:Lanconelli:Priola10} (\textcolor{black}{Argomento da precisare per passare alla $f$, un taglio sfruttando la densit\`a esplicita e le sue propriet\`a}) to assert that the unique bounded solution of Cauchy problem \eqref{eq:OU_initial} actually satisfies
\begin{equation}\label{LP_OU_NON_DEG}
 \Vert D_{x_1}^2u \Vert_{L^p} \, \le \, C_p\Vert f \Vert_{L^p}.
\end{equation}
 We believe that the approach in \cite{Bramanti:Cupini:Lanconelli:Priola10} could extend to show that for \eqref{SOB_EST_GLOB} still hold in this general setting, but such estimates have not been, up to our best knowledge, proven yet.  In this case, to apply Theorem \ref{thm:OU_general} we consider $[\cdot]_1
=\Vert D_{x_1}^2\cdot\Vert_{L^{p}}$ and $
[\cdot]_2=\Vert\cdot\Vert_{L^p}$.

From the above information, namely \eqref{SCHAU_OU}, \eqref{SOB_EST_GLOB}, \eqref{LP_OU_NON_DEG}, and Theorem \ref{thm:OU_general} we obtain the following final result:

\begin{theorem}\label{CONCRETE_PERTURBATION_ESTIMATES} Let $A$, $B$ be two matrices of the form \eqref{eq:Lancon_Pol} such that the Kalman condition \textbf{[K]} holds.  Then, for $f$ in $B_b\left(0,T;C^\infty_0(\R^N)\right)$, the  unique bounded, continuous solution $u\colon [0,T]\times \R^N\rightarrow \R$ in integral form of the associated Cauchy Problem \eqref{eq:OU_initial} which is also smooth in space is such that:
\begin{trivlist}
\item[-] For all $\beta\in (0,1) $, there exists $C_\beta:=C(T,N,\beta)$ such that:
$$ \sup_{t\in[0,T]}\Vert u(t,\cdot)\Vert_{C^{2+\beta}_{b,d}} \, \le \, C_\beta \sup_{t \in (0,T)}\Vert f(t,\cdot) \Vert_{C^\beta_{b,d}}.$$
\item[-] For $p\in (1,+\infty) $, there exists $C_p:=C(T,N,p)$ such that:
$$\Vert D_{x_1}^2u \Vert_{L^p} \, \le \, C_p\Vert f \Vert_{L^p}.$$
If additionally $A_{ij}=0$ for $ j>i$ then
$$\Vert u \Vert_{\dot{W}^{2,p}_d} \, \le \,  C_p \Vert f \Vert_{L^p}.$$
\end{trivlist}
Let now $S:[0,T]\mapsto \R^N\otimes \R^N$ be a continuous matrix valued mapping s.t. for all $t\in [0,T] $, $S(t) $ is symmetric and non-negative. Then, the unique
bounded, continuous solution $u_S\colon [0,T]\times \R^N\rightarrow \R$ in integral form of the associated Cauchy Problem \eqref{eq:OU_PERT} which is also smooth in space also satisfies that:
\begin{trivlist}
\item[-] For all $\beta\in (0,1) $,
$$ \sup_{t\in[0,T]}\Vert u_S(t,\cdot)\Vert_{C^{2+\beta}_{b,d}} \, \le \, C_\beta \sup_{t \in (0,T)}\Vert f(t,\cdot) \Vert_{C^\beta_{b,d}}.$$
\item[-] For $p\in (1,+\infty) $,
$$\Vert D_{x_1}^2u_S \Vert_{L^p} \, \le \, C_p\Vert f \Vert_{L^p}.$$
If additionally $A_{ij}=0$ for $ j>i$ then
$$\Vert u_S \Vert_{\dot{W}^{2,p}_d} \, \le \,  C_p \Vert f \Vert_{L^p},$$
\end{trivlist}
for the \textbf{same previous constants} $C_\beta,C_p$.

\end{theorem}

\section{On some related applications}
In order to illustrate one possible  application of the above estimates we will consider the following stochastic differential equation:
\begin{equation}\label{SDE_ATTRITO}
dX_t=\big(AX_t +F(t,X_t)\big) dt+\left ( \begin{array}{c}\sigma_1(t,X_t) dW_t^1\\
\sigma_2(t) dW_t^2\end{array}\right),
\end{equation}
where $(W_t^1,W_t^2) $ are $d$-dimensional Brownian motions defined on some probability space  $(\Omega,\mathcal F,\mathbb P) $, possibly dependent.  $X_t=(X_t^1,X_t^2)\in \R^{2d} $. We assume that:

\begin{trivlist}
\item[-]  The  matrices
$A$ and $B=\left ( \begin{array}{cc} I_{d,d}&  0_{d,d}\\ 0_{d,d}& 0_{d,d}\end{array}\right) \in \R^{2d}\otimes \R^{2d} $ satisfy the Kalman condition \textbf{[K]}.

\item[-\textbf{[F]}] The non linear drift $F=(F_1,F_2)$, is measurable in time and s.t.
\begin{itemize}
\item[-]$\exists K>0,\ \forall t\in [0,T],\ |b(t,0)|\le K $.
\item[-] $x=(x_1,x_2)\mapsto F_1(t,x) $  $\beta^1 $-H\"older continuous uniformly in $t\in [0,T]$.
\item[-] $x=(x_1,x_2)\mapsto F_2(t,x)=F_2(t,x_2)$ is uniformly $\beta^2>1/3$ H\"older continuous uniformly in $t\in [0,T]$.
\end{itemize}
\item[-\textbf{[UE]}]The diffusion coefficient $a_1=\sigma_1\sigma_1^*$ is uniformly elliptic and bounded, i.e.
there exists $\kappa\ge 1$ s.t. for all $t\in [0,T], x\in \R^{2d},\ \xi\in \R^d $:
\begin{equation}
\label{UE}
\kappa^{-1}|\xi|^2 \le \langle a_1(t,x)\xi,\xi\rangle \le \kappa |\xi|^2.
\end{equation}
\item[-\textbf{[L]}] Local condition for the diffusion. There exists  a measurable function $\varsigma:[0,T]\rightarrow {\mathcal S}_d $ (symmetric matrices of dimension $d$)
satisfying \textbf{[UE]} and such that
\begin{eqnarray}
\label{LOC_COND}
\varepsilon_{a_1}&:=&\sup_{0\le t \le T}\sup_{x\in \R^{2d} } |a_1(t,x)-\varsigma(t)|,
\end{eqnarray}
is \textit{small}. In particular, we do not assume any \textit{a priori} continuity of $a_1$.
\item[-\textbf{[C]}] The mapping $t\mapsto \sigma_2(t) $ is continuous.
\end{trivlist}
Importantly, we do not assume any non-degeneracy condition on the coefficient $\sigma_2$.

We have the following result.
\begin{theorem} \label{THM_SDE_ATTRITO}
Assume \textbf{[H]}, \textbf{[F]}, \textbf{[UE]}, \textbf{[L]}, \textbf{[C]} are in force. Then, the martingale problem associated with $(L_t)_{t\ge 0} $ where for $\varphi\in C_0^\infty(\R^{2d}) $, $(t,x)\in [0,T]\times \R^{2d} $,
\begin{align}
 L_t\varphi(x)=& \frac 12 {\rm Tr}\Big( a_1(t,x) D_{x_1}^2 \varphi(x)\Big)+\frac 12 {\rm Tr}\Big( a_2(t) D_{x_2}^2 \varphi(x)\Big)\notag\\
 &+\langle (Ax+F(t,x)),D\varphi(x)\rangle,\ a_2(t):=\sigma_2\sigma_2^*(t),
 \end{align}
is well posed. In particular, there exists a unique weak solution to equation \eqref{SDE_ATTRITO}. In particular, denoting by $\mathbb P$ the only solution of the martingale problem on $C([0,T],\R^{2d}) $, and denoting by $(X_t)_{t\ge 0}$ the associated canonical process, it holds that for given $p,q$ s.t. $4d/p+2/q<2 $, there exists $C_{p,q}\ge 1$ s.t.  for all $f\in L^q((0,T),L^p(\R^{2d}) $:
$$\mathbb E^{{\mathbb P}_{t,x}}[\int_t^T f(s,X_s)ds]\le C_{p,q}\|f\|_{L^q((0,T),L^p(\R^{2d}))} .$$
\end{theorem}
To prove Theorem \ref{THM_SDE_ATTRITO} we will proceed through a perturbative approach similarly to what was done in \cite{chau:meno:17}, \cite{menozziSPA}. To this end we first consider an underlying Gaussian proxy.
\subsection{Deterministic flow and Proxy Gaussian model}
 Introduce, for fixed $ T>0,\ y \in \R^{2d}$ and $t\in [0,T]$ the backward flow:
\begin{align}
\label{DET_SYST}
\overset{.}{\theta}_{t,T}(y)=&G(t,\theta_{t,T}(y)),\ \theta_{T,T}(y)=y\notag\\
G(t,z)=&Az+F(t,z),\ \in \R^{2d} .
\end{align}
\begin{remark}
We mention carefully that from  the Cauchy-Peano theorem, a solution to  \eqref{DET_SYST} exists. Indeed, the coefficients are continuous and have at most linear growth.
\end{remark}

\subsection{Linearized Multi-scale Gaussian Process and Associated Controls}\label{sec:freezing}
We now want to introduce the forward linearized flow around a solution of \eqref{DET_SYST}. Namely, we consider for $s\ge 0$ the deterministic ODE
\begin{equation}
\label{LIN_AROUND_BK_FLOW}
\overset{.}{\tilde{\phi}}_s =  A
\tilde{\phi}_s
+{ F}(s,\theta_{s,T}(y)).
\end{equation}

We consider now the Gaussian stochastic linearized dynamics $(\tilde X_s^{T,y})_{s\in [t,T]} $:
\begin{eqnarray}
&&\hspace*{-.5cm}d\tilde X_s^{T,y}=[ A\tilde X_s^{T,y}+F(s,\theta_{s,T}(y))]ds +\left(\begin{array}{c}\sigma_1(s,\theta_{s,T}(y)) dW_s^1\\
\sigma_2(s)dW_s^2,\end{array}\right),\nonumber\\
&& \hspace*{.75cm}\forall s\in  [t,T],\
 \tilde X_t^{T,y}=x. \label{FROZ}
 \end{eqnarray}

 From equations \eqref{LIN_AROUND_BK_FLOW}  we explicitly integrate \eqref{FROZ} to obtain for all $v\in [t,T] $:
 \begin{equation}
 \label{INTEGRATED}
\begin{split}
\tilde  X_v^{T,y}&=e^{A(v-t)}x+\int_t^v  e^{A(v-s)} F(s,\theta_{s,T}(y))ds\\
 & +\int_t^v e^{A(v-s)}\left(\begin{array}{c}\sigma_1(s,\theta_{s,T}(y)) dW_s^1\\
 \sigma_2(s)dW_s^2\end{array}\right).
 \end{split}
 \end{equation}
}

\mysection{Additional stability results in anisotropic Sobolev space and Schauder estimates}
\label{ESTENSIONI}

In this section we extend the previous approach to derive the stability with respect to a second order perturbation of the  OU operator in \eqref{DEF_OU_OP_PROXY} under the Kalman condition \textbf{[K]}. Here we consider   also
$L^p$-estimates involving the degenerate components of the OU operator and some associated Schauder estimates.

\subsection{Anisotropic Sobolev spaces and maximal
$L^p$ regularity}
With the notations of Section \ref{KH_OU} we write $z\in \R^N $ as $z=(x,y)=(x,y_1,\cdots, y_k) $ with $x\in\R^{d_0}$, $y_i\in \R^{\mathfrak d_i},\ i\in \{1,\cdots,k \},\ \sum_{i=1}^k \mathfrak d_i=d_1 $.

Given $\beta $ in $(0,1)$ and $i$ in $\llbracket 1, k \rrbracket$, we want to introduce the $\beta$-fractional Laplacian $\Delta^\beta_{y_i}$ along the \textit{component} $y_i$.
To do so, we follow \cite{Huang:Menozzi:Priola19} by considering the orthogonal projection $p_i\colon \R^N\to \R^{\mathfrak d_i}$ such that $p_i(z)=p_i\big((x,y)\big)=y_i$ and denoting its adjoint by $E_i\colon \R^{\mathfrak d_i}\to \R^N$.
\newline
We can now define the $\beta$-fractional Laplacian $\Delta^\beta_{y_i}$ as:
\[\Delta^{\beta}_{y_i}\phi(z) \, := \, \text{p.v.}\int_{\R^{\mathfrak d_i}}\left[\phi(z+E_i w)-\phi(z)\right] \frac{dw}{|w|^{\mathfrak d_i+\textcolor{black}{2}\beta}}, \quad z \in \R^N,\]
for any sufficiently regular function $\phi \colon \R^N\to \R$.\newline
Let $p$ in $(1,+\infty)$, we recall that we have denoted by $L^p((0,T)\times \R^N)$ the standard $L^p$-space with respect to the Lebesgue measure.\newline
We can now define the appropriate anisotropic Sobolev space to state our results.
For notational simplicity, let us denote
\begin{equation}\label{INDEXES}
\alpha_{i} \,:=\, \frac{1}{1+2 i}.
\end{equation}
Set now $\alpha:=(\alpha_1,\cdots,\alpha_k) \in \R^{k}$. The \textit{homogeneous} space $\dot{W}^{2,p}_\alpha([0,T]\times \R^N)$ is composed by all the functions $\varphi\colon [0,T]\times\R^N\to \R$ in $L^p([0,T]\times \R^N)$ such that $(t,z)\in [0,T]\times\R^N\mapsto \Delta_x \varphi(t,z) \in L^p([0,T]\times \R^N) $, where $\Delta_x \varphi$ is intended  in distributional sense,   and for any $i$ in $\llbracket 1, k \rrbracket$, $\Delta^{\alpha_i}_{y_i}\varphi(t,z)$ is well defined for almost every $(t,z)$ and
\[\Delta^{\alpha_i}_{y_i}\varphi(t,z) \, := \,  \Delta^{\alpha_i}_{y_i}\varphi(t,\cdot)(z)   \text{ belongs to } L^p([0,T]\times \R^N).\]
It is endowed with the natural \textit{semi}-norm $\Vert \varphi \Vert_{\dot{W}_\alpha^{2,p}}$  where
 \begin{equation}\label{SEMI_NORM_SOB}
 \Vert \varphi \Vert_{\dot{W}_\alpha^{2,p}}^p \, = \, \|\Delta_x \varphi\|_{L^p}^p +\sum_{i=1}^k\Vert \Delta^{\alpha_i}_{y_i}\varphi \Vert_{L^p}^p.
\end{equation}
The thresholds in \eqref{INDEXES} might seem awkward at first sight. They actually correspond to the indexes needed to get stability of the harmonic functions associated with the principal part of \eqref{DEF_OU_OP_PROXY}, that is considering $A_0$ consisting in the subdiagonal part of $A$ only  (i.e., considering  \eqref{sotto} when the diagonal elements and the  strictly upper diagonal elements  are equal  to zero) along  an associated dilation operator. Namely, setting  \begin{equation}
\label{DEF_OU_OP_PROXY_0}
{L}_0^\text{ou} f(z) = 
{\rm Tr}(B D^2 f(z))
+ \langle A_0 z , D f(z)\rangle,\;\;\;  \; z =(x,y) \in \R^{d_0 + d_1}= \R^N,
\end{equation}
so that $A_0,B $ satisfy \textbf{[K]}, if $ (\partial_t -{L}_0^\text{ou})u(t,z)=0$ then for all $ \lambda>0$
$(\partial_t -\mathcal{L}_0^\text{ou})u\big(\delta_\lambda(t,z) \big)=0$ where the dilation operator
$$\delta_\lambda(t,z)=(\lambda^{1/2}t , \lambda x, \lambda^{1/3} y_1,\cdots, \lambda^{1/(1+2k)} y_k).
$$
precisely exhibits the exponents in \eqref{INDEXES} for the degenerate components.

In \cite{Huang:Menozzi:Priola19}, see also \cite{chen:zhan:19}  and \cite{menozziSPA} where time inhomogeneous coefficients are considered as well, it has been proven that if $A,B $ satisfy \textbf{[K]} and  {\sl  the  diagonal and the strictly upper diagonal elements of $A$ in \eqref{sotto}
are equal to zero} (i.e., $A = A_0$)
then the following Sobolev estimates hold:
\begin{equation}\label{SOB_EST_GLOB}
\Vert u \Vert_{\dot{W}^{2,p}_\alpha} \, \le \, C_p\Vert f \Vert_{L^p},
\end{equation}
 with  $C_p = C_p (\nu, A, d_0, d_1)$,  where again $u$ is the unique bounded solution  to the corresponding Cauchy problem \eqref{eq:OU_initial:intro}. In particular we get also the maximal smoothing effects w.r.t. the degenerate directions.
  Note that the solution $u$ to \eqref{ko_K}
 verifies
\eqref{SOB_EST_GLOB}.
 The specific structure assumed on $A$ is actually due
to the fact that for such matrices there is an underlying homogeneous space structure which makes easier to establish maximal regularity estimates (see e.g. \cite{coif:weis:71} in this general setting).

 If  $A,B $ satisfy \textbf{[K]} with  a general $A$ as in \eqref{sotto}, having non zero strictly upper diagonal entries (non zero entries in the diagonal should  not create difficulties)
 we believe that  the approach in \cite{Bramanti:Cupini:Lanconelli:Priola10} could extend to show that  \eqref{SOB_EST_GLOB} still holds in this general setting. However   such estimates have not been, up to our best knowledge, proven yet.

 \paragraph {\bf $L^p$-estimates for  the degenerate directions of special OU operators.}Setting, as in Section \ref{SEZ_STRAT}, $u(t,z) = v(t, e^{tA} z)$
and since $ u$ solves \eqref{eq:OU_initial:intro} we have that $v$ in turn solves \eqref{ma}. From  the previous computations, setting $B_I = \begin{pmatrix}
 I_{d_0,d_0}  & 0_{d_0,d_1}\\
 0_{d_1,d_0} & 0_{d_1,d_1}
\end{pmatrix}$ and considering $A$ as in \cite{Huang:Menozzi:Priola19},    with  the diagonal and the strictly upper diagonal elements of $A$
 equal to zero in \eqref{sotto}, we derive
\begin{align*}
\|D_x^2 u  \|_{L^p((0,T)\times \R^N)}=&\|  
B_I e^{t A^*} D^2 v (t, e^{tA} \cdot ) \, e^{tA} B_I 
\|_{L^p((0,T)\times \R^N)} \\
\le& C_p  \|  \tilde f(t, e^{tA}  \cdot )  \|_{L^p((0,T)\times \R^N)}.
\end{align*}
On the other hand, for all $i\in \{1,\cdots,k\} $ and with $\alpha_i $ as in \eqref{INDEXES},
\begin{gather*}
 \|\Delta_{y_i}^{\alpha_i} u  \|_{L^p((0,T)\times \R^N)}^p
 \\ =\int_0^T dt\int_{\R^N}dz \Big|{\rm p.v.}\int_{\R^{\mathfrak d_i}}[u(t,z+E_iw)-u(t,z)] \frac{dw}{|w|^{\mathfrak d_i+2\alpha_i}}\Big|^p\\
=\int_0^T dt\int_{\R^N}dz \Big|{\rm p.v.}\int_{\R^{\mathfrak d_i}}[v(t,e^{tA}(z+E_iw))-v(t, e^{tA}z)] \frac{dw}{|w|^{\mathfrak d_i+2\alpha_i}}\Big|^p
\\
 =\int_0^T dt\int_{\R^N}dz \Big|{\rm p.v.}\int_{\R^{\mathfrak d_i}}[v(t, z+ e^{tA} E_iw))-v(t, z)] \frac{dw}{|w|^{\mathfrak d_i+2\alpha_i}}\Big|^p
\\
=:\|\Delta^{\alpha_i,i,A}v\|_{L^p((0,T)\times \R^N)}^p,
\end{gather*}
using that $Tr(A)=0$. Hence, setting
\begin{gather*}
 \|\Delta^{\alpha_0,0,A}v\|_{L^p((0,T)\times \R^N)}^p:= \| {\rm Tr} \big(
 {B_I} e^{ t A^*} D^2 v (t, e^{tA} \cdot ) \, e^{tA}{B_I}\big)  \|_{L^p((0,T)\times \R^N)}^p
\\ =  \| {\rm Tr} \big(B_I e^{ t A^*} D^2 v  \, e^{tA} B_I \big)  \|_{L^p((0,T)\times \R^N)}^p,
\end{gather*}
we get from the definition \eqref{SEMI_NORM_SOB} that the estimate \eqref{SOB_EST_GLOB} rewrites in term of $v$ as:
\begin{equation}
\label{EST_SOB_V}
\Vert v \Vert_{\dot{W}^{2,p,A}_\alpha}^p:=\sum_{i=0}^k \|\Delta^{\alpha_i,i,A}v\|_{L^p((0,T)\times \R^N)}^p\le \tilde C_p \|f\|_{L^p((0,T)\times \R^N)}^p
\end{equation}
with $\tilde C_p = C_p^p.$ We now want to prove that for $w$ solving \eqref{d2}, namely
\begin {equation*}
\begin{cases}
 \partial_t w(t, z)    =  {\text Tr} \big( e^{tA} B
 e^{tA^*} D^2 w(t, z) \big) + {\text Tr} \big( e^{tA} S(t)
 e^{tA^*} D^2 w(t, z) \big)
 \\ \;\;\;\;\;\; \quad \quad + \, \tilde f(t,z),\ (t,z)\in (0,T)\times \R^N,\\
 w(0,z)=0,\ z\in \R^N,
 \end{cases}
\end{equation*}
it also holds that
\begin{equation}
\label{EST_SOB_W}
\Vert w\Vert_{\dot{W}^{2,p,A}_\alpha}^p:=\sum_{i=0}^k \|\Delta^{\alpha_i,i,A}w\|_{L^p((0,T)\times \R^N)}^p  \le \tilde C_p \|f\|_{L^p((0,T)\times \R^N)}^p,
\end{equation}
with the same constants $\tilde C_p$ as  in \eqref{EST_SOB_V}. This can be done through the previous perturbative approach of Section \ref{SEC_PERTURBATIVO} employed to prove Theorem \ref{uno}, which actually gives the expected control for   the second order derivatives contribution of the semi-norm $\Vert \cdot\Vert_{\dot{W}^{2,p,A}_\alpha} $.

For the other contributions and  with the notations of Section \ref{SEC_PERTURBATIVO}, with $Q'(s)=e^{sA}S(s)e^{sA^*}$ and  with $m $ which is the Lebesgue measure on $[0,T] \times \R^N$ (indeed in the present  case $g(t)= {\rm det}(e^{-At}) =1$, for all $t$)   we would get
\begin{gather*}
\Vert \bar v_\epsilon\Vert_{\dot{W}^{2,p,A}_\alpha}^p=\sum_{i=0}^k\|   \Delta^{\alpha_i,i,A}\bar v_{\epsilon}   \,  \|_{L^p((0,T)\times \R^N)}^p
= \sum_{i=0}^k \int_{(0,T) \times \R^N} |\Delta^{\alpha_i,i,A} \bar v_{\epsilon} (t,z) |^p dz dt
\\
=\sum_{i=0}^k \int_{(0,T) \times \R^N} | \E [\Delta^{\alpha_i,i,A}  v_{\epsilon} (t,z + \epsilon X_t) ] \,|^p dz dt
\\
\le\sum_{i=0}^k   \int_{[0,T] \times \R^N} \E [ | \Delta^{\alpha_i,i,A}  v_{\epsilon} (t,z + \epsilon X_t)  \,|^p] dz dt
\\
= \sum_{i=0}^k\E \int_{[0,T] \times \R^N}   | \Delta^{\alpha_i,i,A}  v_{\epsilon}  (t,z + \epsilon X_t)  \,|^p dz dt
\\
 = \sum_{i=0}^k\textcolor{black}{\mathbb E}\int_{[0,T] \times \R^N}  |  \Delta^{\alpha_i,i,A}  v_{\epsilon} (t,\bar z )  \,|^p d\bar z dt
\le \tilde C_p  \|   f  \|_{L^p((0,T)\times \R^N)}^p,
\end{gather*}
{using for the last inequality that $v_\epsilon$ also satisfies \eqref{EST_SOB_V} (similarly to what had been established in \eqref{w11})}.

The same previous procedure  and the
 final  compactness argument then yields \eqref{EST_SOB_W}. Setting eventually $\tilde u(t,z) := w(t, e^{tA} z)$, which is the unique integral solution (smooth in space) of
 \begin{equation*}
\begin{cases}
  \partial_t  u_S(t,z) =  L_t^{\text{ou}, S}  u_S(t,z) + f  (t,z),\ (t,z)\in (0,T)\times \R^N,\\
  u_S(0,z)=0,\ z\in \R^N,
\end{cases}
\end{equation*}
where  $ L_t^{\text{ou}, S}$ introduced in \eqref{ciao}  is the Ornstein-Uhlenbeck operator  perturbed at second order,  we derive that
\begin{equation}\label{SOB_EST_GLOB_PERT}
\Vert  u_S \Vert_{\dot{W}^{2,p}_\alpha} \, \le \, C_p\Vert f \Vert_{L^p},
\end{equation}
with $C_p$ as in \eqref{SOB_EST_GLOB}. We have thus extended the results of Theorem \ref{d44} with the anisotropic Sobolev semi-norm  in \eqref{SEMI_NORM_SOB}. The estimate \eqref{SOB_EST_GLOB} is stable for a continuous, non-negative second order perturbation of the underlying degenerate Ornstein-Uhlenbeck operator.

\subsection{Anisotropic Schauder estimates}
Following Krylov \cite{book:Krylov96_Holder}, for  some fixed $\ell$ in $\N_0:=\N\cup\{0\}$ and $\beta$ in $(0,1]$, we introduce for a function $\phi\colon \R^N\to \R$ the Zygmund-H\"older semi-norm  as
\[[\phi]_{C^{\ell+\beta}} \, := \,
\begin{cases}
    \sup_{\vert \vartheta \vert= \ell}\sup_{x\neq y}\frac{\vert D^\vartheta\phi(x)-D^\vartheta\phi(y)\vert}{\vert x-y\vert^\beta} , & \mbox{if
    }\beta \neq 1; \\
    \sup_{\vert \vartheta \vert= \ell}\sup_{x\neq y}\frac{\bigl{\vert}D^\vartheta\phi(x)+D^\vartheta\phi(y)-2D^\vartheta\phi(\frac{x+y}{2})
    \bigr{\vert}}{\vert x-y \vert}, & \mbox{if } \beta =1
                             \end{cases}\]
(we are using usual multi-indices $\vartheta$ for the partial derivatives).
Consequently, the Zygmund-H\"older space $C^{\ell+\beta}_b(\R^N)$ is the family of bounded functions $\phi\colon \R^N
\to\R$ such that $\phi$ and its derivatives up to order $\ell$ are continuous and the norm
\[\Vert \phi \Vert_{C^{\ell+\beta}_b} \,:=\,  \sum_{i=0}^{\ell}\sup_{\vert\vartheta\vert = i}\Vert D^\vartheta\phi
\Vert_{\infty}+[\phi]_{C^{\ell+\beta}} \,\text{ is finite.}
\]
We can now define  the anisotropic Zygmund-H\"older spaces associated with the current setting and which again reflect the various scales already introduced in \eqref{INDEXES}. Let $\gamma\in (0,3)$, the space $C^{\gamma}_{b,d}(\R^N)$ is
the family of functions $\phi\colon \R^N\to \R$ such that for any $i$ in $\llbracket 0,k\rrbracket$ and any $z_0$ in $\R^N$, the real  function
\[w\in  \R^{\mathfrak d_i}\,  \to \, \phi(z_0+E_i(w))  \,\text{ belongs to }C^{\gamma/(1+2i)}_b\left(\R^{\mathfrak d_i}\right),\]
with a norm bounded by a constant independent from $z_0$. In the above expression, we recall that
 the $(E_i)_{i\in \{1,\cdots,k\}} $ have been defined in the previous paragraph, $\mathfrak d_0=d_0 $ and $E_0$ is the embedding matrix from $\R^{d_0} $ into $\R^N $.
 It is endowed with the norm
\begin{equation}\label{eq:def_anistotropic_norm}
\Vert\phi\Vert_{C^{\gamma}_{b,d}} \,:=\,\sup_{z_0\in \R^N}\Vert\phi\big(z_0+E_0(\cdot)\big)\Vert_{C^{\gamma}_b (\R^{\mathfrak d_0})  }+\sum_{i=1}^{k}\sup_{z_0\in \R^N} \| \phi\big(z_0+ E_i(\cdot)\big)\|_{C^{\gamma/(1+2i)}_b (\R^{\mathfrak d_i} )   }.
\end{equation}
We denote by $C^{\gamma}_{b,d}$ this function space because the regularity exponents reflect again the multi-scale features of the system; the norm could equivalently be defined through the corresponding spatial parabolic distance $d$ defined as follows. For all $z=(x,y),z'=(x',y')\in \R^N$:
$$d(z,z'):=|x-x'|+\sum_{i=1}^{k}\vert y_i-y_i'\vert^{\frac{1}{1+2i}},$$
where the exponents are again those who appeared in \eqref{INDEXES}.
\newline
 Let  $f  \in B_b\left(0,T;C^\infty_0(\R^N)\right)$.
  Under \textbf{[K]},  by the results of Lunardi
 \cite{Lunardi97} it follows that the unique bounded solution
 of the Cauchy Problem \eqref{eq:OU_initial:intro}
 (written in integral form) verifies the
    following anisotropic Schauder estimates
\begin{equation}\label{SCHAU_OU}
\Vert u \Vert_{L^\infty((0,T),C^{2+\beta}_{b,d})} \, \le \, \textcolor{black}{C_\beta} \Vert f \Vert_{L^\infty((0,T),C^{\beta}_{b,d})},
\end{equation}
for some constant ${C_\beta}$ independent from $f$, i.e.,
\begin{equation}\label{s22}
\sup_{0 \le t \le T}\Vert u (t, \cdot )\Vert_{C^{2+\beta}_{b,d}} \, \le \, \textcolor{black}{C_\beta}  \sup_{0 \le t \le T }\Vert f (t, \cdot )\Vert_{C^{\beta}_{b,d}}.
\end{equation}
We again set as in the previous paragraph $u(t,z) = v(t, e^{tA} z)$
\begin{align}
\Vert u \Vert_{L^\infty((0,T),C^{2+\beta}_{b,d})}=\Vert v(t,e^{tA}\cdot) \Vert_{L^\infty((0,T),C^{2+\beta}_{b,d})}=:\Vert v \Vert_{L^\infty((0,T),C^{2+\beta}_{b,d,A})}\notag\\
\le \textcolor{black}{C_\beta} \Vert f \Vert_{L^\infty((0,T),C^{\beta}_{b,d})}
 =   {C_\beta}\Vert \tilde f(t,e^{tA}\cdot) \Vert_{L^\infty((0,T),C^{\beta}_{b,d})}=:  {C_\beta}\Vert \tilde f\Vert_{L^\infty((0,T),C^{\beta}_{b,d,A})},\label{SC_EST_V}
\end{align}
denoting $\tilde{f}(t,z):= f(t,e^{-tA}z)$.
  We again want to prove as in Section 1.1 %
that for $w$ solving \eqref{d2},
\begin{align}\label{STIMA_W_PER_SC}
\Vert w \Vert_{L^\infty((0,T),C^{2+\beta}_{b,d,A})}\le \textcolor{black}{C_\beta}\Vert \tilde f\Vert_{L^\infty((0,T),C^{\beta}_{b,d,A})}
\end{align}
with the same  constant $\textcolor{black}{C_\beta}$ as in \eqref{SC_EST_V}. We proceed  through the previous perturbative approach of Section \ref{SEC_PERTURBATIVO}.
With the notations employed therein, we deduce that there exists a unique  solution $v_\epsilon=$  PDE$(Q,\tilde f(t,z-\epsilon X_t))$, depending also on $\epsilon$ and  $\omega$ as parameters  such that
\begin{align}
\label{proof:estimate1_SC}
\sup_{(t,z)\in [0,T]\times \R^N}| v_\epsilon(t,z)| \, &\le \, T \sup_{(t,z)\in [0,T]\times \R^N}|\tilde  f(t,z)|.
\end{align}
By  the translation  invariance of the H\"older-norms,  using also that  $X_{t}=$ $ e^{tA} e^{-tA} X_t$,  it is not difficult to prove that, for any $\omega$, $\P$-a.s.,
\begin{equation}\label{wqs}
\|  \tilde f  \|_{L^\infty((0,T),C^{\beta}_{b,d,A})} = \|  \tilde f(\cdot , \cdot \, -\epsilon X_{\cdot}) \|_{L^\infty((0,T),C^{\beta}_{b,d,A})} .
\end{equation}
Thus it also holds from \eqref{SC_EST_V}
\begin{equation}\label{w11_BIS}
\|  v_{\epsilon}    \|_{L^\infty((0,T),C^{2+\beta}_{b,d,A})} \le \textcolor{black}{C_\beta}  \|  \tilde f  \|_{L^\infty((0,T),C^{\beta}_{b,d,A})}.
\end{equation}
Recalling now that $\bar v_\epsilon(s,z) = \mathbb{E}[v_\epsilon(s,z+\epsilon X_s)]$
 is   an integral solution of
\begin{equation*}
    \partial_t \bar v_\epsilon(t,z) \, = \, \text{tr}(Q(t)D^2_z \bar v_\epsilon(t,z))+\lambda\left(\bar v_\epsilon(t,z+\epsilon l(t))-\bar v_\epsilon(t,z)\right)+ %
    \tilde f(t,z),
\end{equation*}
with zero initial condition, we write that for $i\in \{1,\cdots, k\} $, $w,w'\in \R^{\mathfrak d_i} ,$ $ \ (t,z_0)\in [0,T]\times \R^N$,
\begin{align*}
& %
|\bar v_\epsilon(t,e^{At}(z_0+E_i(w))-\bar v_\epsilon(t,e^{At}(z_0+E_i(w'))|\\
\le&\E[ |v_\epsilon(t,e^{At}(z_0+E_i(w))+\epsilon  e^{At}e^{-At} X_t)
\\ &- v_\epsilon(t,e^{At}(z_0+E_i(w'))+\epsilon e^{At}e^{-At}  X_t)|]
\\
\le& \E[ [v_\epsilon(t,e^{At}(z_0+E_i(\cdot)))]_{C^{\frac{2+\beta}{1+2i}}} ] \; |w-w'|^{\frac{2+\beta}{1+2i}}.
\end{align*}
 Hence,
\begin{align*}
[\bar v_\epsilon(t,e^{At}(z_0+E_i(\cdot)))]_{C^{\frac{2+\beta}{1+2i}}}\le& \E[ [v_\epsilon(t,e^{At}(z_0+E_i(\cdot)))]_{C^{\frac{2+\beta}{1+2i}}} ].
\end{align*}
We would get, similarly,
\begin{align*}
[D_x^2\bar v_\epsilon(t,e^{At}(z_0+E_0(\cdot)))]_{C^{\beta}}\le& \E[ [D_x^2v_\epsilon(t,e^{At}(z_0+E_0(\cdot)))]_{C^{\beta}} ],
\end{align*}
and for all $k\in \{1,2\} $,
\begin{align*}
\|D_x^k\bar v_\epsilon(t,e^{At}(z_0+E_0(\cdot))) \|_\infty\le \E[ \|D_x^k v_\epsilon(t,e^{At}(z_0+E_0(\cdot))) \|_\infty].
\end{align*}
Summing those contributions, we thus derive from \eqref{eq:def_anistotropic_norm}, \eqref{SC_EST_V} that:
\begin{equation}
\|\bar v_\epsilon\|_{L^\infty((0,T),C^{2+\beta}_{b,d,A})}\le \sup_{0 \le t \le T}\E[\|v_\epsilon (t, \cdot)\|_{C^{2+\beta}_{b,d,A})}]  \le \textcolor{black}{C_\beta}  \|  \tilde f  \|_{L^\infty((0,T),C^{\beta}_{b,d,A})},
\end{equation}
using \eqref{w11_BIS} for the last inequality.
  Now, continuing   as in Section 3.2, using also  a compactness argument, one would derive
 that \eqref{STIMA_W_PER_SC} indeed holds.

Going  backwards, setting  $\tilde u(t,z) := w(t, e^{tA} z)$,   we find that  $\tilde u  $  is the unique (integral) solution $u_S$ to \eqref{eq:OU_PERT0};
 we finally derive that
\begin{equation}\label{SCHAU_OU_PERT}
\Vert  u_S \Vert_{L^\infty((0,T),C^{2+\beta}_{b,d})} \, \le \, \textcolor{black}{C_\beta} \Vert f \Vert_{L^\infty((0,T),C^{\beta}_{b,d})},
\end{equation}
where $C_{\beta}$ is the same constant as in \eqref{SCHAU_OU}. Estimate  \eqref{SCHAU_OU_PERT} provides the extension of Theorem \ref{d44} for the anisotropic Schauder estimates.

\vskip 1mm
\begin{remark}
Let us mention that for the perturbative method to work, roughly speaking, few properties were actually needed on the underlying norm. Namely, we used the translation invariance and some kind of commutation between the norm (or a function of the norm in the $L^p$-case) and  expectation. Hence, this approach could  possibly be applied to a much wider class of estimates in other function spaces (like e.g. Besov spaces). This will concern further research. \qed
 \end{remark}

\section*{Acknowledgment}
For the first author, the work was supported by a public grant ($2018-0024$H) as part of the FMJH project. For the second author, the article was prepared within the framework of the HSE University Basic Research Program. The third author
  has been partially supported by   GNAMPA of the Istituto Nazionale di Alta  Matematica (INdAM).

\end{document}